\documentclass[12pt,twoside]{article}
\usepackage{amssymb,amsmath,amsthm}
\usepackage{enumerate}
\usepackage[dvips]{epsfig}

\usepackage{multicol}
\newtheorem{definition}{Definition}[section]
\newtheorem{theorem}{Theorem}[section]
\newtheorem{proposition}{Proposition}[section]
\newtheorem{lemma}{Lemma}[section]

\newtheorem{remark}{Remark}[section]
\DeclareMathOperator{\cn}{div}

 \newcommand{\N}{\mbox{\rm I$\!$N}}
 \newcommand{\R}{\mbox{\rm I$\!$R}}
\newcommand{\RR}{\mathbb{R}}
\newcommand{\grad}{{\nabla}}
\newcommand{\C}{\mathcal{C}}

\def\runninghead#1#2{\pagestyle{myheadings}
  \markboth{{\protect\footnotesize\it\centerline{#1}}\hfill}
  {\hfill{\protect\footnotesize\it\centerline{#2}}}}
\def\qed{\hbox{${\vcenter{\vbox{			
  \hrule height 0.4pt\hbox{\vrule width 0.4pt height 6pt
  \kern5pt\vrule width 0.4pt}\hrule height 0.4pt}}}$}}

\newcommand{\TT}{\ensuremath{\mathcal{T}}}

\newcommand{\dx}{\ensuremath{\, dx}}

\newcommand{\dt}{\ensuremath{\, dt}}

\def\O{\Omega}

\def\lV{\left\lVert}
\def\rV{\right\rVert}

\def\h{{\mathcal H}}
\def\N {\textit{{\tiny N}}}
\def \sign {{Z}}

\def\Delt{{\scriptstyle\Delta}}
\def\dist{{\rm dist\,}}
\def\ph{{\varphi}}
\def\sign{{\rm sign\,}}
\def\meas{{\rm meas\,}}
\def\div{\mbox{$\mathrm{div}$}}

\def\ptI{{\!,}}
\def\ptK{{K}}
\def\ptL{{L}}
\def\K{{K}}
\def\L{{L}}

\def\KIL{{{\textstyle\sigma}_{\!\!\K,\!\L}}}

\def\ptKIL{{\!\ptK\ptI\ptL}}

\def\mK{{|K|}}

\def\diam{{\rm diam\,}}

\def\ptI{,}
\def\bel{\begin{equation}\label}
\def\ee{\end{equation}}
\def\ptTau{{h}}
\def\delt{{{\scriptstyle\Delta} t}}
\def\Frond{{\vec{\mathcal F}}}

\def\Tau{{\TT}}

\def\QKn{{Q_\ptK^n}}

\def\dsp{\displaystyle}

\def\size{{\rm size\,}}
\def\diam{{\rm diam\,}}

\def\ptl{{\scriptscriptstyle \partial}}

\newcommand{\charu}{{1\hspace*{-3.5pt}1}}
\def\ptl{{\partial}}

\DeclareSymbolFont{AMSb}{U}{msb}{m}{n}
\DeclareSymbolFontAlphabet{\mathbb}{AMSb}
\def\Bbb{\mathbb}
\newcommand{\Rset}{\ensuremath{{\Bbb R}}}

\newcommand{\norm}[1]{\ensuremath{\left\|#1\right\|}}
\newcommand{\abs}[1]{\ensuremath{\left|#1\right|}}
\newcommand{\Grad}{\mathrm{\nabla}}
\newcommand{\disp}{\ensuremath{\displaystyle}}

\newcommand{\sn}{\sum_{n=0}^{N-1}}

\begin{document}

\runninghead{Bendahmane, Khalil, Saad}{Finite volume scheme for water gas}
\thispagestyle{empty}

\begin{center}
{\large {\bf CONVERGENCE OF A FINITE VOLUME SCHEME FOR GAS WATER FLOW IN A MULTI-DIMENSIONAL POROUS MEDIA}} 
\hspace{2cm}
\medskip

MOSTAFA BENDAHMANE$^1$,  ZIAD KHALIL$^2$, MAZEN SAAD$^2$ \\
\end{center}
{\footnotesize \it
\begin{center}
$^1$Institut Mathematiques de Bordeaux, Universite Victor Segalen Bordeaux 2, 
Place de la Victoire, 33076 Bordeaux, France.\\
 $^2$ Ecole Centrale de Nantes, Universit\'e de Nantes,
Laboratoire de Math\'ematiques Jean Leray, UMR CNRS 6629,
1, rue de la No\'e, 44321 Nantes France\\
E-mail :mostafa.bendahmane@u-bordeaux2.fr, Ziad.Khalil@ec-nantes.fr,  Mazen.Saad@ec-nantes.fr \\
~\\
{Research partially supported by GNR MOMAS}
\end{center}
}
\bigskip
~~
\bigskip

\noindent{\bf Abstract.}
A classical model for water-gas flows in porous media is considered. The degenerate coupled system of equations obtained by mass conservation is usually approximated by finite volume schemes in the oil reservoir simulations. The convergence properties of these schemes  are only known for incompressible fluids. This chapter deals with construction and   convergence analysis of a finite volume scheme for compressible and immiscible flow in porous media. In comparison with incompressible fluid, compressible fluids requires more powerful techniques. We present a new result of convergence in a two or three dimensional porous medium and under the only modification that the density of gas depends on global pressure.

\bigskip

\section{Introduction}
\label{ch4:secintro}
\setcounter{equation}{0}

Mathematical study of a petroleum-engineering schemes takes an important place in oil recovery engineering for production of hydrocarbons from
petroleum reservoirs. In soil mechanics, engineers study the air-water flow in soils and they prefer the use of a two phase flow model. More recently, due to the effects of global warming on climate change,
two and multi-phase flow has been receiving an increasing attention in connection with the disposal of radioactive waste and sequestration of $CO_2$.

It has been shown that the governing equations describing two incompressible (compressible) phase flow in porous media can be written in a fractional flow formulation, i.e., in terms of global pressure and saturation and that formulation has been studied; For incompressible flow and from a mathematical point of view  \cite{A-AD85,B-CJ86} and it has  been used in numerical codes \cite{CCJDD,CCJEGW,CCJLB}.
For immiscible and compressible two-phase flows (e.g., air, water), Ewing and al. in (\cite{ewing1}, \cite{chen2}) follow the ideas of
Chavent by considering global pressure and saturation as unknowns of the system. This formulation leads to a global pressure equation coupled to the water
saturation equation. The authors proposed a finite element  and finite difference method to solve the saturation equation and
a mixed finite element to approximate the global pressure equation. Note that the global pressure reads  as a parabolic equation
with a source term involving the evolution of the capillary pressure term. This evolution term is approached by Picard
iterations. This algorithm converges numerically  and suggests a continuous dependence on  the  capillary
terms and legitimates some approximations for small capillary pressure.
Further, it has been proven that this fractional flow approach is far more efficient than the original two-pressure approach from the computational point of view \cite{chen} and the references cited therein. For  compressible flow and from mathematical  point of view, the fractional flow formulation  is sufficient enough at least for slightly compressible gas, i.e, when the density of gas depends on the global pressure \cite{A-GS08,A-GS08-2}. More recently and under the context of theoretical study of compressible flow in porous media, the two-pressure approach
has been treated by Z. Khalil, M. Saad \cite{1-KS10,2-KS10}.\\
In this paper, we consider immiscible two-phase flows; the gas phase is considered to be compressible and the water one to be incompressible. The model is derived by using the global pressure notation and is justified at least for slightly compressible gas. The system represents two kinds of degeneracy. The first one is the classical degeneracy of the diffusion operator in saturation due to the capillary effect. The second one represents a degeneracy in the evolution term in pressure occurring in the region where the gas saturation vanishes: A classical compactness result on pressure is missed in the region where the gas phase is missing.

The aim of the present paper is to show that the approximate solutions obtained with the proposed upwind finite volume scheme (\ref{ch4:prob-discr1})-(\ref{ch4:prob-discr2}), converges as the mesh size tends to zero, to a solution of system (\ref{ch4:gas})-(\ref{ch4:water}) in an appropriate sense defined in Section \ref{ch4:sec:formulation}. In Section \ref{ch4:sec:FVS} we introduce the finite volume discretization, the numerical scheme and state the main convergence results. In Section \ref{ch4:sec:basic-apriori}, maximum principle on saturation is attained and {\it  a priori } estimates on the discrete gradient of the capillary term  and on the discrete gradient of the global pressure are derived as the continuous case  in C. Galusinski, M. Saad \cite{A-GS08-2}. In Section \ref{ch4:sec-exist-est}, a well posed-ness of the scheme is inspired by H. W. Alt, S. Luckhaus\cite{A-AL83}.  Section \ref{ch4:sec-compactness} is devoted to a space-time $L^1$  compactness argument, in this section we follow B. Andreianov, M. Bendahmane and R. Ruiz-Baier \cite{}. Finally, the passage to the limit on the scheme needs a powerful techniques due to the lack of compactness result on global pressure in the region where the saturation of gas vanishes, and this performed in section \ref{ch4:sec:conv}.

\date{}

\section{The mathematical formulation}
\label{ch4:sec:formulation}
\setcounter{equation}{0}
The fractional flow formulation  describing the immiscible displacement of two compressible and incompressible
fluids are given by the following mass conservation of each phase \cite{A-GS08-2}:
\begin{multline}
\hspace{1.5cm}\partial_{t}( \phi\rho(p)s) - \cn ({\bf K}\rho(p) M_1(s) \nabla p)
-\cn({\bf K}\rho(p)\alpha (s)\nabla s)\\
+ \cn({\bf K}\rho^2(p) M_1(s) {\bf g})+ \rho(p)sf_{P}^{~}= 0,~~~~~~~~~~~~~\label{ch4:gas}
\end{multline}
\begin{align}
&\partial_{t}(\phi s) + \cn ({\bf K}M_2(s) \nabla p)
-\cn({\bf K}\alpha (s)\nabla s) +  \cn({\bf K}\rho_2 M_2(s) {\bf g})+sf_{P}^{~}= f_{P}^{~}-f_{I}^{~}.\label{ch4:water}
\end{align}

where $\phi$ and ${\bf K}$ are  the porosity and absolute permeability of the porous medium; $\rho, \rho_2, p$ and  $s$ are
respectively the densities of gas and water (density of water is  constant), the global pressure and the saturation of gas;
$f_{P}^{~}, f_{I}^{~}, M_1, M_2$ and  ${\bf g }$ are respectively the production and injection source terms, the mobilities of gas and water and the gravity term.\\
To define the capillary term $\alpha$, let us denote by $p_1, p_2$ to be respectively  the pressures of gas and water phases. Thus, we define the capillary pressure and the total mobility as
\begin{align}
\label{ch4:pc}
&p_{12}(s(t,x)) = p_{1}(t,x) - p_{2}(t,x)\\
\label{ch4:totalmob}
&M(s)=M_1(s)+ M_2(s)
\end{align}
and the function $s\mapsto p_{12}(s)$ is non-decreasing
$(\frac{dp_{12}}{ds}(s) \geq 0, \mbox{ for all } s \in [0,1])$.
In this paper, the forced displacement of fluids is modellized. It is
used in many enhanced recovery processes: a fluid, such as water,
is injected into some wells in a reservoir while the resident
hydrocarbons are produced from other wells.
Now we define the capillary term
$$\alpha (s)=  \frac{M_1(s)M_2(s)}{M(s)}\frac{dp_{12}}{ds}(s)\geq
0$$

defining  a function $\tilde{p}(s)$ such that
$\frac{d\tilde{p}}{ds}(s) = \frac{M_1(s)}{M(s)}\frac{dp_{12}}{ds}(s)$,
and setting
$p = p_{2}+\tilde{p}$, named global pressure \cite{B-CJ86}.
 Thus, each phase velocity given by Darcy's law can be written as
\begin{align}
{\bf V}_{1} &=-{\bf K}M_1(s)\nabla p - {\bf K} \alpha(s) \nabla s +{\bf K}M_1(s)\rho_1(p){\bf g}\label{ch4:V1}\\
{\bf V}_{2} &=-{\bf K}M_2(s)\nabla p + {\bf K} \alpha(s) \nabla s  +{\bf K}M_2(s)\rho_2{\bf g}.\label{ch4:V2}
\end{align}

Note that this system is strongly degenerate. In fact,
the lack of coercivity of the degenerate
diffusion term $\cn({\bf K}\alpha (s)\nabla s)$ is classical for
incompressible flows. An additional difficulty is due to the degeneracy
of the time derivative term $\phi(x) \partial_{t}( \rho(p)s)$ which vanishes
in the region where $s=0$.
Another difficulty seems to be the degenerate diffusive pressure term
$\cn ({\bf K}\rho(p) M_1(s) \nabla p)$ in (\ref{ch4:gas}) where the mobility of the gas phase $M_1$
vanishes in $s=0$. In fact, a pressure diffusion term appears also in
the saturation equation (\ref{ch4:water}) with the term $\cn ({\bf
  K}M_2(s) \nabla p)$.
An energy estimate coupling the two equations (\ref{ch4:gas})-(\ref{ch4:water})
lets appear a non-degenerate dissipative pressure term
(see section \ref{ch4:sec:basic-apriori}).

Consider a fixed time $T>0$ and let $\Omega$ be a bounded set of $\Rset^d$ ($d\ge 1$).
We set $Q_T=(0,T)\times\Omega$, $\Sigma_T=(0,T)\times\partial\Omega$.
To the system (\ref{ch4:gas})-(\ref{ch4:water}) we add the following mixed boundary
conditions and initial conditions. We consider the boundary
$\partial \Omega=\Gamma_{w} \cup \Gamma_{i} $, where  $\Gamma_{w}$
denotes the water injection boundary and $\Gamma_{i}$ the impervious one.
\begin{equation} \label{ch4:cl}
\left \{
\begin{aligned}
& s(t,x) = 0, \;\; p(t,x) = 0 \text{ on } \Gamma_{w}\\[0.5ex]
& {\bf V}_1\cdot{\bf n} = {\bf V}_2\cdot{\bf n} = 0 \text{ on }\Gamma_{i},
\end{aligned}
\right .
\end{equation}
where ${\bf n}$ is the outward normal to the boundary $\Gamma_{i}$.
We force a constant pressure (shifted at zero) along the time on the
 region of water injection.

\noindent{\it Initial condition:}
\begin{equation}\label{ch4:ci}
\left\{
\begin{aligned}
& s(0,x) = s_{0}(x),\mbox{ in } \O \\[0.5ex]
& p(0,x) = p_0(x) \mbox{ in } \O
\end{aligned}
\right.
\end{equation}

We are going to construct a finite volume scheme on orthogonal admissible mesh, we treat here the case where
$$
K=k\mathcal{I}_d
$$
where $k$ is a constant positive. For clarity, we take $k=1$ which equivalent to change the scale in time.
In remark \ref{ch4:remarqueK} we give the scheme where $k$ is a function depending on space.

Next we introduce some physically relevant assumptions on
the coefficients of the system.

\bigskip

\begin{enumerate}[{(H}1{)}]
\item \label{ch4:hyp:H1} The porosity $\phi$ belongs to $L^{\infty}(\Omega)$ and
there exist two positive constants $\phi_{0}$ and $\phi_{1}$ such that
$\phi_0 \le \phi(x)\le
\phi_1$ a.e. $x\in\Omega$.
\item \label{ch4:hyp:H3} The functions $M_1$ and $M_2$ belong to ${\mathcal
	C}^0([0,1]; \Rset^+)$, $M_1(0)=0$ 	and $M_2(1)=0$.
	In addition, there is a positive constant $m_0$  such that,
	for all $s\in [0,1]$,
	$$
	M_1(s)+M_2(s) \ge m_0 .
	$$
\item \label{ch4:hyp:H4} The function $\alpha \in {\mathcal C}^2([0,1];\Rset^+)$ satisfies $\alpha(s)>0$ for $0<s\le 1$, and
      $\alpha(0)=0$.

	We define $\beta(s)=\int_0^s\alpha(z)dz$ and assume that $\beta ^{-1}$ is an H\"older function
	of order $\theta$, with $0<\theta \le 1$ on $[0,\beta(1)]$.
	This means that there exists a positive $c$ such that for all $s_1,s_2 \in [0,\beta(1)]$, one has
	$|\beta ^{-1}(s_1)-\beta ^{-1}(s_1)|\le c |s_1-s_2|^\theta$.
\item \label{ch4:hyp:H5} $(f_{P}^{~},f_{I}^{~})\in (L^2(Q_T))^2$, $f_{P}^{~}(t,x)$, $f_{I}^{~}(t,x) \ge 0$
a.e. $(t,x)\in Q_T$
\item  \label{ch4:hyp:H6} The density $\rho$ is a ${\mathcal C}^1(\Rset)$ function, increasing, and
there exist $\rho_m>0$, $\rho_M<+\infty$ such that\\
 $\rho_m\le \rho(p)\le \rho_M$ for all $p\in \Rset$.
\end{enumerate}

\bigskip

The assumptions ({H}\ref{ch4:hyp:H1})--({H}\ref{ch4:hyp:H6}) are classical for porous media. Especially,
a practical sufficient condition to handle ({H}\ref{ch4:hyp:H4}) is to consider that
$\alpha$ is an H\"older function at ${s=0}$. This contains several relevant physical cases of two-phase flows in porous media (see \cite[chapter $V$]{B-CJ86}).

\smallbreak
Define
$$
H^1_{\Gamma_w}(\Omega)= \{\, u\in H^1(\Omega) \,;\, u=0 \text{ on } \Gamma_w \,\},
$$
this is an Hilbert space when equipped with the norm
$\lV u\rV_{H^1_{\Gamma_w}(\Omega)} =\lV\nabla u\rV_{(L^2(\Omega))^d}$.

In the next section we introduce the existence of solutions to
system (\ref{ch4:gas})-(\ref{ch4:water}) under the conditions ({H}\ref{ch4:hyp:H1})--({H}\ref{ch4:hyp:H6}).
\begin{definition}\label{ch4:thmain}(Weak solutions)\\
Let (H\ref{ch4:hyp:H1})-(H\ref{ch4:hyp:H6}) hold. Assume $p_0$ (defined by $(\ref{ch4:ci})$)
belongs to $L^2(\Omega)$, and $s_0$ satisfies $0\le s_0\le 1$
almost everywhere in $\Omega$. Then, the pair $(s,p)$ is a weak solution of the problem (\ref{ch4:gas})-(\ref{ch4:water}) if
\begin{align*}
0\le s\le 1 \text{ a.e. in } Q_T,
~~\beta(s)\in L^2(0,T;H^1_{\Gamma_w}(\Omega)),~~
p\in  L^2(0,T;H^1_{\Gamma_w}(\Omega)),
\end{align*}
such that for all $\varphi,\xi\in \mathcal{D}\big([0,T)\times\Omega\big)$,
\begin{multline}
 -\int_{Q_T}  \phi \rho(p) s \partial_t\varphi\,dxdt
-\int_\Omega \phi (x) u_0(x)\varphi(0,x)\,dx \\
+\int_{Q_T}\rho(p)
M_1(s) { \nabla} p\cdot\nabla \varphi \,dxdt
+\int_{Q_T}\rho(p)\nabla
\beta(s)\cdot\nabla \varphi\,dxdt\\
-\int_{Q_T} \rho^2(p)
M_1(s) {\bf g}\cdot\nabla \varphi \,dxdt
+\int_{Q_T}\rho(p)s f_{P}^{~} \varphi\,dxdt=0,
\label{ch4:pmain}
\end{multline}

\begin{multline}
- \int_{Q_T}  \phi s \partial_t\xi\,dxdt-\int_\Omega \phi s_0(x) \xi(0,x)\,dx
+\int_{Q_T}{ \nabla} \beta(s)\cdot\nabla \xi \,dxdt\\
-\int_{Q_T}M_2(s){\nabla} p\cdot\nabla \xi \,dxdt
-\int_{Q_T}\rho_2 M_2(s){\bf g}\cdot\nabla \xi \,dxdt\\
+\int_{Q_T} s f_{P}^{~} \xi \,dxdt
 = \int_{Q_T} (f_{P}^{~}-f_{I}^{~})\xi \,dxdt. \label{ch4:smain}
\end{multline}
\end{definition}
\section{The finite volume scheme}
\label{ch4:sec:FVS}
\setcounter{equation}{0}

Now, we want to discretize the problem \eqref{ch4:gas}-\eqref{ch4:water}.
Let $\mathcal{T}$ be a regular and admissible mesh of the domain $\Omega$, constituting of open and
convex polygons called control volumes with maximum size (diameter) $h$.

We let $\Omega$ be an open bounded polygonal connected subset of
$R^3$ with boundary $\partial \Omega$. Let $\mathcal{T}$ be an admissible mesh of
the domain $\Omega$ consisting of open and convex polygons called control
volumes with maximum size (diameter) $h$.
For all $K \in \mathcal{T}$, let by $x_K$ denote the center of $K$, $N(K)$ the set of the
neighbors of $K$ i.e. the set of cells of $\mathcal{T}$ which have a common
interface with $K$, by $N_{\text{int}}(K)$ the set of the neighbors of $K$
located in the interior of $\mathcal{T}$, by $N_{\text{ext}}(K)$ the set of
edges of $K$ on the boundary $\partial \Omega$. \\
Furthermore, for all $L \in N_{int}(K)$
denote by $d_{K,L}$ the distance between $x_K$ and $x_{L}$, by $\sigma_{K,L}$
the interface between $K$ and $L$, by $\eta_{K,L}$ the unit normal vector
to $\sigma_{K,L}$ outward to $K$. And for all $\sigma \in N_{\text{ext}}(K)$, denoted by
$d_{K,\sigma}$ the distance from $x_K$ to $\sigma$.

\begin{figure}[ht]
\begin{center}
\includegraphics[width=0.4\linewidth]{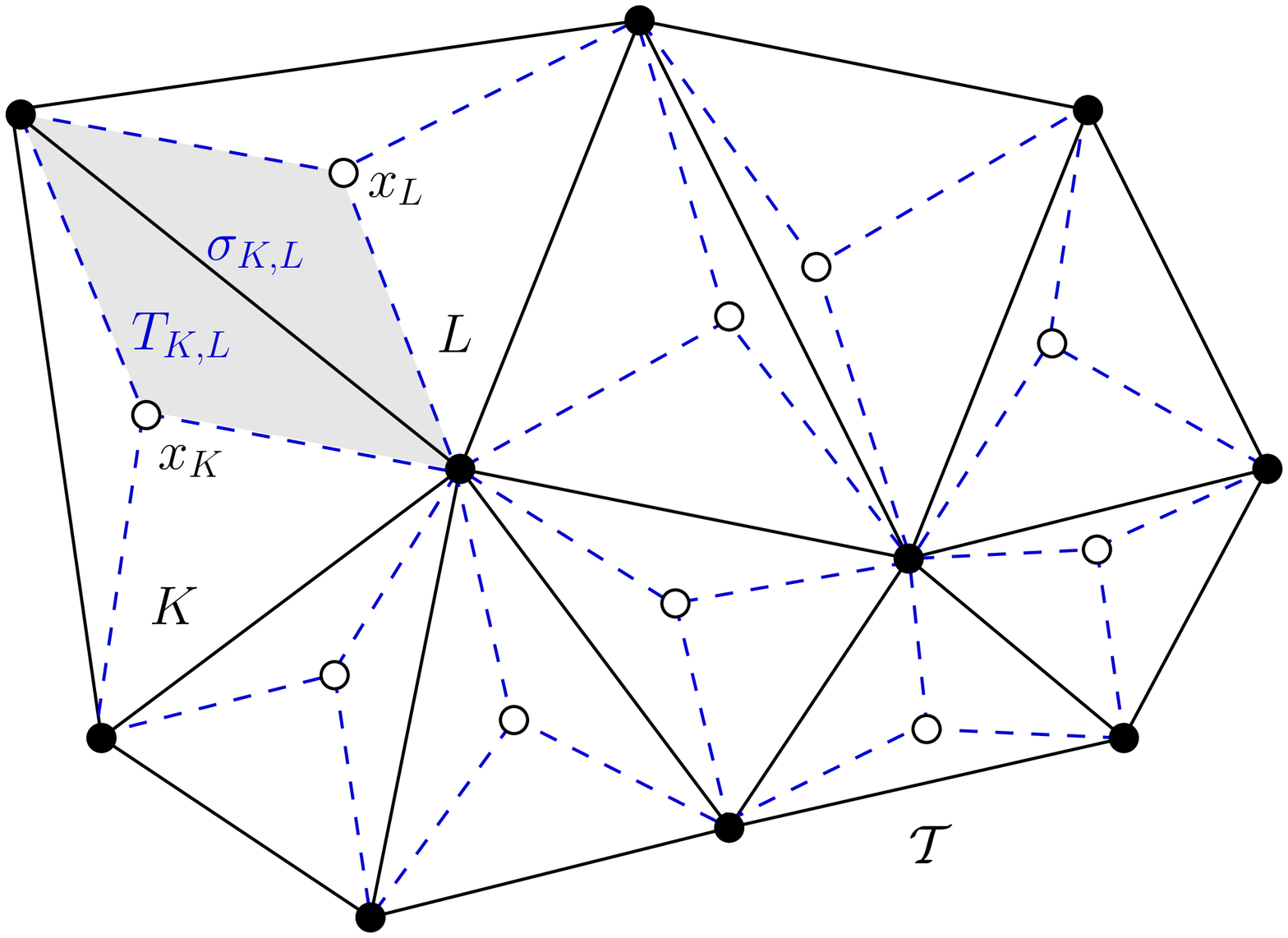}
\end {center}
\caption{\footnotesize
Control volumes, centers and diamonds}
\label{ch4:fig:diamond}
\end{figure}

For all $K \in \mathcal{T}$, we denote by
$\abs{K}$ the measure of $K$. The admissibility of $\mathcal{T}$ implies that
$\overline{\Omega}=\cup_{K\in \mathcal{T}} \overline{K}$, $K\cap L=\emptyset$
if $K,L\in \mathcal{T}$ and $K \ne L$, and there exist a finite sequence
of points $(x_{K})_{K\in \mathcal{T}}$ and the straight line $\overline{x_{K}x_{L}}$
is orthogonal to the edge $\sigma_{K,L}$.
We also need some regularity on the mesh:
$$
\min_{K\in \mathcal{T},L \in N(K)}
\frac{d_{K,L}}{\text{diam}(K)}\ge \alpha
$$
for some $\alpha \in \R^+$.

We denote by $H_h(\Omega) \subset L^2(\Omega)$ the space of functions
which are piecewise constant on each control volume $K \in \mathcal{T}$. For all
$u_h \in H_h(\Omega)$ and for all $K \in \mathcal{T}$, we denote by $U_K$ the constant value
of $u_h$ in $K$. For $(u_h,v_h)\in (H_h(\Omega))^2$, we define the following inner product:
$$
\left \langle u_h,v_h \right \rangle_{H_h}= l\sum_{K \in \mathcal{T} }\sum_{L \in N(K) }
\frac{\abs{\sigma_{K,L}}}{d_{K,L}}(U_{L}-U_{K})(V_{L}-V_K).
$$
In the case of homogeneous Neumann boundary condition, for example $\nabla u\cdot \eta=\nabla v\cdot \eta= 0$ on $\Gamma_i\subset \partial \Omega$, so we impose $U_{L}-U_{K} = V_{L}-V_{K} =0$ if $\sigma_{K,L}\subset \Gamma_i$. And in the case of homogeneous Dirichlet  boundary condition $u = v=0$ on $\Gamma_w\subset \partial \Omega$, so we impose $U_L=V_L=0$ if $\sigma_{K,L}\subset \Gamma_w$ and $d_{K,L}$ denotes the distance form $x_K$ to $\sigma_{K,L}$, more precisely,
$$
\left \langle u_h,v_h \right \rangle_{H_h}= l\sum_{K \in \mathcal{T} }\sum_{L \in N_{int}(K) }
\frac{\abs{\sigma_{K,L}}}{d_{K,L}}(U_{L}-U_{K})(V_{L}-V_K) + l\sum_{K \in \mathcal{T} }\sum_{\sigma \in N_{ext}(K)\cap\Gamma_w }
\frac{\abs{\sigma}}{d_{K,\sigma}}U_{K}V_K.
$$
We define a norm in $H_h(\Omega)$ by
$$
\norm{u_h}_{H_h(\Omega)}=(\left \langle u_h,u_h \right \rangle_{H_h})^{1/2}.
$$
Finally, we define $L_h(\Omega) \subset L^2(\Omega)$ the space of functions
which are piecewise constant on each control volume $K \in \mathcal{T}$ with the associated
norm
$$
\left (u_h,v_h \right )_{L_h(\Omega)}= \sum_{K \in \mathcal{T} }
\abs{K}U_{K} V_K,\qquad \norm{u_h}^2_{L_h(\Omega)}
=\sum_{K \in \mathcal{T} }\abs{K}\abs{U_{K}}^2,
$$
for $(u_h,v_h)\in (L_h(\Omega))^2$.

Next, we let $K \in \mathcal{T}$ and $L \in N(K)$ with common vertexes
$(a_{\ell,K,L})_{1\le \ell\le I}$ with $I \in \N^{\star}$. Next let
$T_{K,L}$ (respectivley $T^{\text{ext}}_{K,\sigma}$ for $\sigma\in N_{\text{ext}}(K)$) be the
open and convex polygon with vertexes $(x_K,x_L)$ ($x_K$ respectively) and
$(a_{\ell,K,L})_{1\le \ell\le I}$. Observe that
$$
\displaystyle \Omega=\cup_{K\in \mathcal{T}}
\Biggl(\Bigl(\cup_{L\in N(K)}\overline{T}_{K,L}\Bigl)
\cup \Bigl(\cup_{\sigma \in N_{\text{ext}}(K)}\overline{T}^{\text{ext}}_{K,\sigma}\Bigl)\Biggl)
$$
The discrete gradient $\Grad_h u_h$ of a constant per control volume function $u_h$ is defined as the constant per diamond $T_{K,L}$ $\R^l$-valued function with values
$$
\Grad_h u_{h}(x)=\begin{cases}
& l \frac{U_{L}-U_{K}}{d_{K,L}}\eta_{K,L}\text{ if $x \in T_{K,L}$},\\
&l \frac{U_{\sigma}-U_{K}}{d_{K,\sigma}}\eta_{K,\sigma}\text{ if $x \in T^{\text{ext}}_{K,\sigma}$},
\end{cases}
$$
Notice that : \\
$\bullet$ The $l$-dimensional mesure $\abs{T_{K,L}}$ of $T_{K,L}$ equals to $\frac 1 l \abs{\sigma_{K,L}}d_{K,L}$.\\
$\bullet$ The semi-norm $\|u_h\|_{H_h}$ coincides with the $L^2(\Omega)$ norm of $\Grad_h u_{h}$, in fact
\begin{multline*}
\|\Grad_h u_h\|_{L^2(\Omega)}^2 =  \sum_{K \in \mathcal{T} }\sum_{L \in N(K) }\int_{T_{K,L}}|\nabla_h u_h|^2\,dx =
l^2 \sum_{K \in \mathcal{T} }\sum_{L \in N(K) }\abs{T_{K,L}} \frac{|U_{L}-U_{K}|^2}{|d_{K,L}|^2}
 \\
 = l \sum_{K \in \mathcal{T} }\sum_{L \in N(K) }\frac{\abs{\sigma_{K,L}} }{d_{K,L}}|U_{L}-U_{K}|^2
 = \norm{u_h}_{H_h(\Omega)}^2
 \end{multline*}
 $\bullet$
 Let $\vec{F}_{K,L}$ for an arbitrary $\R^l$ vector associated with the interface $\sigma_{K,L}$ satisfying $\vec{F}_{K,L}= \vec{F}_{L,K}$. We denote by $\mathcal{E}_h$ the set of interfaces $\sigma_{K,L}$. Then, a discrete field $(\vec{F}_{K,L})_{\sigma_{K,L}\in\mathcal{E}_h}$  is assimilated to the piecewise constant vector function
 $$
 \vec{F}_h = \sum_{\sigma_{K,L}\in\mathcal{E}_h} \vec{F}_{L,K} \chi_{T_{K,L}}.
 $$
 The discrete divergence of the field $\vec{F}_h$ is defined as the discrete function $w_h=div_h\vec{F}_h $ with the entires
 $$
div_K \vec{F}_h :=\frac{1}{\abs{K}} \sum_{L\in N(K)} |\sigma_{K,L}|\vec{F}_{K,L}\cdot \eta_{K,L}.
 $$
 A key point of the analysis of the two-point finite volume schemes is the following kind of discrete duality property :\\
 \begin{lemma}\label{ch4:lam:stokes}
For all discrete function $w_h$ on  $\Omega$  which is null on $\partial\Omega$, for all discrete field $\vec{F}_{h}$  on  $\Omega$,
$$
 \sum_{K\in \mathcal{T}} \abs{K}w_K div_K \vec{F}_{h} = - \sum_{\sigma_{K,L}\in \mathcal{E}_h}\abs{T_{K,L}}\nabla_{K,L}w_h \cdot\vec{F}_{K,L}.
 $$
 \end{lemma}

 \begin{proof}
 \begin{align*}
 \sum_{K\in \mathcal{T}} \abs{K} w_K div_K \vec{F}_{h}
 &=\sum_{K\in \mathcal{T}}\sum_{L\in N(K)}\abs{\sigma_{K,L}} w_K \vec{F}_{K,L}\cdot \eta_{K,L}\\
 &=-\frac{1}{2}\sum_{K\in \mathcal{T}}\sum_{L\in N(K)}\abs{\sigma_{K,L}}    (w_L-w_K) \eta_{K,L} \cdot\vec{F}_{K,L}\\
 &=-\frac{1}{2}\sum_{K\in \mathcal{T}}\sum_{L\in N(K)}\frac 1 l \abs{\sigma_{K,L}} d_{K,L}\,\,  l\frac{ (w_L-w_K)}{d_{K,L} } \eta_{K,L} \cdot\vec{F}_{K,L} \\
 &=- \sum_{\sigma_{K,L}\in \mathcal{E}_h}\abs{T_{K,L}}\nabla_{K,L}w_h \cdot\vec{F}_{K,L} .
\end{align*}

 \end{proof}
Next, we approximate $M_i(s)\Grad p\cdot\eta_{K,L},~~ (i=1,2.)$ by means of the values $s_K,s_L$ and $p_K,p_L$
that are available in the neighborhood of the interface $\sigma_{K,L}$. To do this,
let us use some function $G_i$ of $(a,b,c)\in\R^3$.
The numerical convection flux functions $G_i$,
$G_i\in C(\R^3,\R)$, satisfies the following properties:
\begin{equation}\label{ch4:Hypfluxes}
\begin{cases}
\text{(a) $G_i(\cdot,b,c)$ is non-decreasing for all $b,c\in\R$,}\\
\hspace*{15pt} \text{and $G_i(a,\cdot,c)$ is non-increasing
 for all $ a,c\in \R$};\\
\text{(b) $G_1(a,a,c)=-M_1(a)\,c$ and $G_2(a,a,c)=M_2(a)\,c$ for all $a,c \in \R$};
\\ \text{(c) $G_i(a,b,c)=-G_i(b,a,-c)$
and there exists $C>0$ such that }\\
 \,\,\quad|G_i(a,b,c)|\leq C\,
\bigl(|a|+|b|\bigr)|c| \text{ for all $a,b,c \in \R$}\\
 \text{(d) There exists a constant $m_0$ such that}\\
 \text{ $(G_2(a,b,c)-G_1(a,b,c))c \ge m_0|c|^2$
 for all $ a,b,c \in \R$}.
\end{cases}
\end{equation}

\begin{remark}
Note that  the assumptions (a), (b) and (c) are standard and they respectively ensure the maximum principle on saturation, the consistency of the numerical flux, and the conservation of the numerical flux on each interface. Moreover, the last assumption (d) will be used to obtain the $L^2$ estimate of discrete gradient of the pressure $p$. Practical examples of numerical convective flux functions can be found in \cite{EyGaHe:book}.

In our context, one possibility to construct the numerical
flux $G_i$ satisfying \eqref{ch4:Hypfluxes} is to split $M_i$ in the
non-decreasing part ${M_i}_{\uparrow}$ and the non-increasing part
${M_i}_{\downarrow}$:
$$
{M_i}_\uparrow(z):=\int_0^z ({M_i}'(s))^+\,ds \quad
{M_i}_\downarrow(z):=-\int_0^z ({M_i}'(s))^-\,ds.
$$
Herein, $s^+=\max(s,0)$ and $s^-=\max(-s,0)$. Then we take
$$
G_i(a,b;c)
={c}^+\Bigl({M_i}_\uparrow(a)+{M_i}_\downarrow(b)\Bigr)-{c}^-\Bigl({M_i}_\uparrow(b)+{M_i}_\downarrow(a)\Bigr),
$$
which leads to,

\begin{equation}\label{ch4:}
\begin{cases}
G_1(a,b;c)=-M_1(b)\,{c}^+ - (-M_1(a))\,{c}^-\\
G_2(a,b;c)=M_2(b)\,{c}^+  -M_2(a)\,{c}^-
\end{cases}
\end{equation}
Note that the function $s\mapsto M_1(s)$ is non-decreasing, and the function $s\mapsto M_2(s)$ is non-increasing which lead to the monotony property of the function $G_i$. Furthermore, depending on the assumption (H\ref{ch4:hyp:H3}) on the total mobility we have,
\begin{align}\label{ch4:tmob}
\Big( G_2(a,b,c)-G_1(a,b,c)\Big) c=M(b){c^+}^2 + M(a){c^-}^2\geq m_0 c^2.
\end{align}
\end{remark}

The next goal is to discretize the problem (\ref{ch4:gas})-(\ref{ch4:water}).
We denote by $\mathcal{D}$ an admissible discretization of $Q_T$, which consists of
an admissible mesh of $\Omega$, a time step $\Delta t>0$, and a positive number $N$ chosen
as the smallest integer such that $N\Delta t\ge T$.
We set,
 \begin{align*}
t^n&:=n \Delta t
&&\text{ for } n\in  \{0,\ldots,N\}\\
dp^{n+1}_{K,L}&:=\frac{\abs{\sigma_{K,L}}}{d_{K,L}}(p^{n+1}_{L}-p^{n+1}_{K})
&&\text{ for } n\in  \{0,\ldots,N-1\}\\
\rho^{n+1}_{K,L}&:=\frac{1}{p_L^{n+1}-p_K^{n+1}}\int_{p_K^{n+1}}^{p_L^{n+1}} \rho(\zeta)\,d\zeta
&&\text{ for } n\in  \{0,\ldots,N-1\}\\
f^{n+1}_{P,K}&:=\frac{1}{\Delta t \abs{K}}\int_{t^n}^{t^{n+1}}\int_K f_p(x)\,dx dt
&&\text{ for } n\in  \{0,\ldots,N-1\}\\
f^{n+1}_{I,K}&:=\frac{1}{\Delta t \abs{K}}\int_{t^n}^{t^{n+1}}\int_K f_I(x)\,dx dt
&&\text{ for } n\in  \{0,\ldots,N-1\}\\
{\bf g}_{K,L}&:=\int_{K/L}({\bf g}\cdot \eta_{K,L})^+\, d\gamma(x)=\int_{K/L}({\bf g}\cdot \eta_{L,K})^-\, d\gamma(x)
&&
\end{align*}

A finite volume  scheme for the discretization of the problem
(\ref{ch4:gas})-(\ref{ch4:water})
is giving by the following set of equations with unknowns
$P=(p_{K}^{n+1})_{K \in \mathcal{T}},n\in[0,N]$ and
$S=(s_{K}^{n+1})_{K \in \mathcal{T}},n\in[0,N]$, for
all $K \in \mathcal{T}$ and $n \in [0,N]$
\begin{equation}\label{ch4:prob:init}
p_K^0=\frac{1}{\abs{K}} \int_{K} p_0(x) \,dx,\qquad
s_K^0=\frac{1}{\abs{K}} \int_{K} s_0(x) \,dx,
\end{equation}
and
\begin{multline}\label{ch4:prob-discr1}
\abs{K}\phi_K\frac{\rho(p^{n+1}_K)s^{n+1}_K-\rho(p^{n}_K)s^{n}_K}{\Delta t}
-\sum_{L \in N(K) }
\frac{\abs{\sigma_{K,L}}}{d_{K,L}}\rho^{n+1}_{K,L}(\beta(s^{n+1}_{L})-\beta(s^{n+1}_{K}))
\\
 + \sum_{L \in N(K) } \rho^{n+1}_{K,L}G_1(s^{n+1}_K,s^{n+1}_{L};dp^{n+1}_{K,L}) +F^{(n+1)}_{1,K}+ \abs{K}\rho(p^{n+1}_K)s^{n+1}_K f_{P,K}^{n+1}=0,
\end{multline}
\begin{multline}\label{ch4:prob-discr2}
\abs{K}\phi_K\frac{s^{n+1}_K-s^{n}_K}{\Delta t}- \sum_{L \in N(K) }
\frac{\abs{\sigma_{K,L}}}{d_{K,L}}(\beta(s^{n+1}_{L})-\beta(s^{n+1}_{K}))\\
+ \sum_{L \in N(K) } G_2(s^{n+1}_K,s^{n+1}_{L};dp^{n+1}_{K,L}) + F^{(n+1)}_{2,K}+ \abs{K}(s^{n+1}_K-1) f_{P,K}^{n+1}=
-\abs{K} f_{I,K}^{n+1},
\end{multline}
where $\disp F^{n+1}_{1,K}$ the approximation of $\disp \int_{\partial K}\rho^2(p^{n+1}) M_1(s^{n+1}) {\bf g}\cdot \eta_{K,L}\, d\gamma(x)$ by an upwind scheme :
\begin{equation}
 F^{n+1}_{1,K}=\sum_{L \in N(K) }F^{n+1}_{1,K,L}=\sum_{L \in N(K) }\Big(\rho^2(p^{n+1}_K) M_1(s^{n+1}_K){\bf g}_{K,L}
- \rho^2(p^{n+1}_L) M_1(s^{n+1}_L){\bf g}_{L,K}\Big),
\end{equation}
and  similarly $\disp F^{n+1}_{2,K}$ the approximation of $\disp\int_{\partial K}\rho_2 M_2(s^{n+1}) {\bf g}\cdot \eta_{K,L}\, d\gamma(x)$ such that
\begin{align}
 F^{(n+1}_{2,K}=\sum_{L \in N(K)} F^{(n+1}_{2,K,L}=\sum_{L \in N(K)}\Big(\rho_2 M_2(s^{n+1}_L){\bf g}_{K,L}
- \rho_2 M_2(s^{n+1}_K){\bf g}_{L,K}\Big).
\end{align}
Note that the numerical fluxes to approach the gravity terms $F_1, F_2$ are nondecreasing with respect to $s_K$ and nonincreasing with respect $s_L$.

We extend the mobility functions $s\mapsto M_1(s)$ and $s\mapsto M_2(s)$ outside $[0,1]$ by continues constant functions as follows,
The approximate solutions, $p_{\delta t,h},~s_{\delta t,h}: \R^+\times \Omega \to \R$
given for all $K \in \mathcal{T}$ and $n \in [0,N]$ by

\begin{equation}\label{ch4:prob:general}
p_{\delta t,h}(t,x)=p_{K}^{n+1} \text{ and }s_{\delta t,h}(t,x)=s_K^{n+1},
\end{equation}

for all $x \in K$ and $t \in (n \Delta t,(n+1)\Delta t)$.

\medskip
The main result of this paper is the following theorem.
\begin{theorem} \label{ch4:theo1}
Assume that (H\ref{ch4:hyp:H1})-(H\ref{ch4:hyp:H6}) hold. Let $(p_{0},s_{0})\in L^2(\Omega, \R)\times L^\infty(\Omega,\R)$ and $0\le s_0\le 1$ a.e. in $\Omega$.
Then there exists an approximate
solution $(p_{\delta t,h},s_{\delta t,h})$ to the system
(\ref{ch4:prob-discr1})-(\ref{ch4:prob-discr2}), which converges (up to a subsequence)
to $(p,s)$ as $(\delta t,h) \to (0,0)$, where $(p,s)$ is a weak solution
to the system (\ref{ch4:gas})-(\ref{ch4:water}) in the sense of the Definition \ref{ch4:thmain}.
\end{theorem}

\section{A priori estimates}\label{ch4:sec:basic-apriori}
We are now concerned with a uniform estimate on the discrete gradient of $\beta(s)$, and on the discrete gradient of the global  pressure $p$.
\subsection{Nonnegativity}
We aim to prove the following lemma which is a basis to the analysis that we are going to perform.
\begin{lemma}
Let $(s_K^{0})_{K \in \mathcal{T}}\in [0,1]$.
Then, the solution $(s_K^{n})_{K \in \mathcal{T},n \in  \{0,\ldots,N\}}$,  of the finite volume scheme \eqref{ch4:prob:init}-\eqref{ch4:prob-discr2} remains in  $[0,1]$.
\end{lemma}

\begin{proof}
Let us show by induction in $n$ that  for all $K \in \mathcal{T}, ~ s^n_K\geq 0$. The claim is true for $n=0$ and for all $K \in \mathcal{T}$. We argue by induction that for all $K \in \mathcal{T}$, the claim is true up to order $n$. We consider the control volume $K$ such that $s^{n+1}_K=\min{\{s^{n+1}_L\}}_{L\in \mathcal{T}}$, and we seek that $s^{n+1}_K\geq 0$.\\
For the above mentioned purpose, multiply the equation in \eqref{ch4:prob-discr1} by $-(s_{K}^{n+1})^-$, we obtain
  \begin{align}
  &-\abs{K}\phi_K\frac{\rho(p^{n+1}_K)s^{n+1}_K-\rho(p^{n}_K)s^{n}_K}{\Delta t}(s_{K}^{n+1})^- \notag\\
&+\sum_{L \in N(K) }
\frac{\abs{\sigma_{K,L}}}{d_{K,L}}\rho^{n+1}_{K,L}(\beta(s^{n+1}_{L})-\beta(s^{n+1}_{K}))(s_{K}^{n+1})^-
\notag\\
& - \sum_{L \in N(K) } \rho^{n+1}_{K,L}G_1(s^{n+1}_K,s^{n+1}_{L};dp^{n+1}_{K,L})(s_{K}^{n+1})^- \notag\\
 &-F^{(n+1)}_{1,K}(s_{K}^{n+1})^-- \abs{K}\rho(p^{n+1}_K)s^{n+1}_K f_{P,K}(s_{K}^{n+1})^-=0,\label{ch4:nonnegat:I}
   \end{align}
Observe that $\beta(s^{n+1}_{L})-\beta(s^{n+1}_{K})\geq 0$ (recall that $\beta$ is nondecreasing). Which implies
\begin{equation}\label{ch4:nonnegat:I:1}\begin{split}
  &\sum_{L \in N(K)}
  \frac{\abs{\sigma_{K,L}}}{d_{K,L}}(\beta(s^{n+1}_{L})-\beta(s^{n+1}_{K}))(s_{K}^{n+1})^-\geq 0.
\end{split}\end{equation}
The numerical flux $G_1$ is nonincreasing with respect to $s_L^{n+1}$ (see (a) in \eqref{ch4:Hypfluxes}), and consistence (see (c) in \eqref{ch4:Hypfluxes}), we get
\begin{equation}\begin{split}\label{ch4:nonnegat:I:2}
    G_1(s^{n+1}_K,s^{n+1}_{L};dp^{n+1}_{K,L})\,(s_{K}^{n+1})^-
    &\le G_1(s^{n+1}_K,s^{n+1}_{K};dp^{n+1}_{K,L})\,(s_{K}^{n+1})^-\\
    &=dp^{n+1}_{K,L}\,M_1(s^{n+1}_{K})\,(s_{K}^{n+1})^-=0.
\end{split}\end{equation}
Using the identity $s_{K}^{n+1}=({s_{K}^{n+1}})^+-(s_{K}^{n+1})^-$, and the mobility $M_1$ extended by zero on $]-\infty, 0]$, then $M_1(s^{n+1}_K) (s_{K}^{n+1})^- = 0$ and
\begin{multline}\label{ch4:nonnegat:I:3}
-F^{(n+1)}_{1,K}(s_{K}^{n+1})^-- \abs{K}\rho(p^{n+1}_K)s^{n+1}_K f_{P,K}^{n+1}(s_{K}^{n+1})^-\\
=\sum_{L \in N(K) } \rho^2(p^{n+1}_L) M_1(s^{n+1}_L){\bf g}_{L,K}(s_{K}^{n+1})^-
+\abs{K}\rho(p^{n+1}_K) f_{P,K}|(s_{K}^{n+1})^-|^2\geq 0.
\end{multline}
Then, we deduce from \eqref{ch4:nonnegat:I} that
$$
\abs{K}\phi_K\frac{\rho(p^{n+1}_K)|(s_{K}^{n+1})^-|^2+\rho(p^{n}_K)s^{n}_K(s_{K}^{n+1})^-}{\Delta t}\le 0,
$$
and from the nonnegativity of $s^{n}_{K}$, we obtain $(s_{K}^{n+1})^-=0$.
This implies that  $s_{K}^{n+1}\geq 0$ and
$$
0\le s^{n+1}_K \le s^{n+1}_L \text{ for all } n \in [0,N-1] \text{ and }L \in \mathcal{T}.
$$

To prove that $ s^{n+1}_K \leq 1 \text{ for all}\, n \in [0,N-1]\text{ and}\, K \in \mathcal{T}$.
 We argue by induction that for all $K\in \mathcal{T}$, $s^n_K\leq 1$.
Let the control volume $K$ such that $s^{n+1}_K=\max{\{s^{n+1}_L\}}_{L \in \mathcal{T}}$, and let us show that $s^{n+1}_K\leq 1$.\\
For the mentioned claim, we multiply the equation in \eqref{ch4:prob-discr2} by $(s_{K}^{n+1}-1)^+$,
\begin{multline}\label{ch4:nonnegat:IE}
\abs{K}\phi_K\frac{s^{n+1}_K-s^{n}_K}{\Delta t}(s_{K}^{n+1}-1)^+ - \sum_{L \in N(K) }
\frac{\abs{\sigma_{K,L}}}{d_{K,L}}(\beta(s^{n+1}_{L})-\beta(s^{n+1}_{K}))(s_{K}^{n+1}-1)^+\\
+ \sum_{L \in N(K) } G_2(s^{n+1}_K,s^{n+1}_{L};dp^{n+1}_{K,L})(s_{K}^{n+1}-1)^+ + F^{(n+1)}_{2,K}(s_{K}^{n+1}-1)^+ \\
+ \abs{K}(s^{n+1}_K -1)f_{P,K}(s_{K}^{n+1}-1)^+= -\abs{K} f_{I,K}^{n+1}(s_{K}^{n+1}-1)^+
\end{multline}
Since $\beta$ is nondecreasing, we get $\beta(s^{n+1}_{L})-\beta(s^{n+1}_{K})\leq 0$. This implies
\begin{equation}\label{ch4:nonnegat:IE:1}
- \sum_{L \in N(K) }
\frac{\abs{\sigma_{K,L}}}{d_{K,L}}(\beta(s^{n+1}_{L})-\beta(s^{n+1}_{K}))(s_{K}^{n+1}-1)^+\geq0.
\end{equation}
Next, we use the fact that the numerical flux $G_2$ is nondecreasing with respect to $s_{K}^{n+1}$ and consistence (see $(b)$ and $(c)$ in (\ref{ch4:Hypfluxes}) to deduce
\begin{equation}\begin{split}\label{ch4:nonnegat:IE:2}
    G_2(s^{n+1}_K,s^{n+1}_{L};dp^{n+1}_{K,L})\,(s_{K}^{n+1}-1)^+
    &\geq G_2(s^{n+1}_K,s^{n+1}_{K};dp^{n+1}_{K,L})\,(s_{K}^{n+1}-1)^+\\
    &=dp^{n+1}_{K,L}\,M_2(s^{n+1}_{K})\,(s_{K}^{n+1}-1)^+=0,
\end{split}\end{equation}
now, we rely on the extension of the mobility $M_2$ by zero on $[1, \infty[$, thus  $M_2(s^{n+1}_K)$ $ (s_{K}^{n+1}-1)^+=0$, to deduce
\begin{equation}\label{ch4:nonnegat:IE:3}
 F^{(n+1)}_{2,K} (s_{K}^{n+1}-1)^+
=\sum_{L \in N(K)}\rho_2 M_2(s^{n+1}_L){\bf g}_{K,L} (s_{K}^{n+1}-1)^+ \geq 0
\end{equation}
It is clear that the production source term in the left hand side of (\ref{ch4:nonnegat:IE}) is nonnegative and the injection source term on the right hand side is nonpositive.\\
Using the above estimates to deduce from (\ref{ch4:nonnegat:IE}) that,
\begin{multline} \label{ch4:s1}
\abs{K}\phi_K\frac{s^{n+1}_K-s^{n}_K}{\Delta t}(s_{K}^{n+1}-1)^+=
\\\frac{\abs{K}\phi_K}{\Delta t}\Big((s^{n+1}_K-1)(s_{K}^{n+1}-1)^+
-(s^{n}_K-1)(s_{K}^{n+1}-1)^+\Big)\leq0
\end{multline}
Using again the identity $(s_{K}^{n+1}-1)=({s_{K}^{n+1}}-1)^+-(s_{K}^{n+1}-1)^-$, and that\\ $s_{K}^{n}\leq1$ to deduce from \eqref{ch4:s1} that $({s_{K}^{n+1}}-1)^+=0$.  Consequently, we obtain
$$
 s^{n+1}_L \le s^{n+1}_K \leq 1 \text{ for all}\, n \in [0,N-1]\text{ and}\, L \in \mathcal{T}.
 $$
\end{proof}

\subsection{Discrete a priori estimates}
Let us recall the following two lemmas ::
\begin{lemma}\label{ch4:lem-disc-sob}(Discrete Poincar\'{e} inequality)\cite{EyGaHe:book} \\
Let $\O$ be an open bounded polygonal subset of $\R^d$, $ d=2~ \text{or}~ 3 $, $\mathcal{T}$ an admissible finite volume mesh in the sense given in
 the  section \ref{ch4:sec:FVS}, and let $u$ be a function which is constant on each cell $K \in \mathcal{T}$, that is,
 $u(x)=u_K$ if $x\in K,$ then
$$\norm{u}_{L^2(\O)}\leq diam(\O) \norm{u}_{H_h(\O)},$$
where $\norm{\cdot}_{H_h(\O)}$ is the discrete $H_0^1$ norm.
\end{lemma}
\begin{remark}(Dirichlet condition on part of the boundary) This \it{lemma} gives a discrete Poincar\'{e} inequality for Dirichlet  boundary conditions
on the boundary $\partial \O$. In the case of Dirichlet condition on part of the boundary only, it is still possible to prove a discrete Poincar\'{e} inequality
provided that the polygonal bounded open set $\O$ is  connected.
\end{remark}
\begin{lemma}\label{ch4:lemm-integ-parts}(Discrete integration by parts formula)
Let $\O$ be an open bounded polygonal subset of $\R^d$, $\mathcal{T}$ an admissible finite volume mesh in the sense given in
 the  subsection \ref{ch4:sec:FVS}. Let $F_{K/L},~ K\in \mathcal{T}$  and $L\in N(K)$ be a value in $\R$ depends on $K$ and $L$ such that $F_{K/L}=-F_{L/K}$, and let $\varphi$ be a function which is constant on each cell $K \in \mathcal{T}$, that is,
$\varphi(x)=\varphi_K$ if $x\in K,$ then
\begin{equation}
\label{ch4:pppF}
\sum_{K\in \mathcal{T}}\sum_{L\in N(K)}F_{K/L} \varphi_K=-\frac{1}{2}\sum_{K\in \mathcal{T}}\sum_{L\in N(K)}F_{K/L}(\varphi_L-\varphi_K)
\end{equation}

Consequently, if $F_{K/L}= a_{K/L}(b_L-b_K)$, with $a_{K/L}=a_{L/K}$, then
\begin{equation}
\label{ch4:pppB}
\sum_{K\in \mathcal{T}}\sum_{L\in N(K)}a_{K/L}(b_L-b_K)\varphi_K=-\frac{1}{2}\sum_{K\in \mathcal{T}}\sum_{L\in N(K)}a_{K/L}(b_L-b_K)(\varphi_L-\varphi_K)
\end{equation}
\end{lemma}
\begin{proof}
The sum $\sum_{K\in \mathcal{T}}\sum_{L\in N(K)}$ can be reorganized by edge. In fact, on each edge $\sigma_{K,L}$ between the mesh $K$ and $L$, there are two contributions : from $K$ to $L$ named $F_{K,L}\varphi_K$ and from $L$  to $K$ named $F_{L,K}\varphi_L$, then
\begin{equation}
\label{ch4:pppF1}
\sum_{K\in \mathcal{T}}\sum_{L\in N(K)}F_{K/L} \varphi_K=  \sum_{\sigma_{K,L}} (F_{K/L}\varphi_K + F_{L/K}\varphi_L)
\end{equation}
Using now the fact that $F_{K/L}$ is antisymmetric, then we have
\begin{equation}\label{ch4:pppF2}
\begin{split}
\sum_{K\in \mathcal{T}}\sum_{L\in N(K)}F_{K/L} \varphi_K&=  \sum_{\sigma_{K,L}} F_{K/L}(\varphi_K - \varphi_L)\\
&= \frac 1 2 \sum_{\sigma_{K,L}} \Big(F_{K/L}(\varphi_K - \varphi_L) +  F_{L/K}(\varphi_L - \varphi_K)\Big).
\end{split}
\end{equation}
Finally, reorganise the last summation on edge by mesh, we obtain exaclty \eqref{ch4:pppF}.
The equality \eqref{ch4:pppB} is a direct consequence of \eqref{ch4:pppF}  $F_{K/L}= a_{K/L}(b_L-b_K) = a_{K/L}(b_K-b_L) = -F_{L/K}$.
\end{proof}

We  derive in the next proposition, the main uniform estimates on the discrete gradient of the capillary term $\beta(s)$ and the discrete gradient of the global pressure $p$.
\begin{proposition}\label{ch4:prop:LPBV}
Let $(p_{K}^{n},s_{K}^{n})_{K \in \mathcal{T},n \in
\{0,\ldots,N\}}$, be a solution of the finite volume scheme
\eqref{ch4:prob-discr1}-\eqref{ch4:prob-discr2}. Then, there exist a constant $C>0$, depending
on $\Omega$, $T$, $s_{0}$, $p_{0}$ and $\alpha$ such that
\begin{equation}\label{ch4:est:grad-norm}
\begin{split}
&
\sum_{K\in \mathcal{T}} \abs{K}s_K^N \h(p_K^N) -\sum_{K\in \mathcal{T}} \abs{K}s_K^0 \h(p_K^0) \\
&
+\frac{c_1}{2}  \sum_{n=0}^{N-1}\Delta t\sum_{K\in \mathcal{T}}\sum_{L\in N(K)}
\frac{\abs{\sigma_{K,L}}}{d_{K,L}}
\abs{p_{K}^{n+1}-p_{L}^{n+1}}^2 \le C
\end{split}\end{equation}
and
\begin{equation}\label{ch4:est:grad-norm1}\begin{split}
&
\sum_{K\in \mathcal{T}} \abs{K}B(s_K^N)  -\sum_{K\in \mathcal{T}} \abs{K}B(s_K^0) \\
&
+\frac{1}{4} \sum_{n=0}^{N-1}\Delta t\sum_{K\in \mathcal{T}}\sum_{L\in N(K)}
\frac{\abs{\sigma_{K,L}}}{d_{K,L}}
\abs{\beta(s_{K}^{n+1})-\beta(s_{L}^{n+1})}^2 \le C
\end{split}\end{equation}
where $\disp B^\prime(s)= \beta(s) $, and $\disp \h(p)=g(p)+ \rho(p)p$ with $\disp g^\prime(p)=-\rho(p)$.
\end{proposition}

\begin{proof}
To prove the estimate \eqref{ch4:est:grad-norm}, we multiply the gas discrete equation \eqref{ch4:prob-discr1} and the water discrete equation \eqref{ch4:prob-discr2} respectively  by $p_K^{n+1}, \, g(p^{n+1}_{K})= \h(p^{n+1}_{K})-\rho(p^{n+1}_{K})p^{n+1}_{K}$ and adding them, then summing
the resulting equation over $K$ and $n$, and this yields to,

\begin{equation}\label{ch4:disc-est}
  E_{1}+E_{2}+E_{3}+E_{4}+E_{5}=0
  \end{equation}
where
\begin{align*}
 &E_{1} = \sn \sum_{K \in \mathcal{T}} \abs{K}\phi_K \Big( (\rho(p^{n+1}_K)s^{n+1}_K-\rho(p^{n}_K)s^{n}_K)\; p_K^{n+1}
+ (s^{n+1}_K-s^{n}_K)\; g(p_K^{n+1})\Big),\\
& E_{2} = -\sn \Delta t \sum_{K \in \mathcal{T}}
\sum_{L \in N(K)}\frac{\abs{\sigma_{K,L}}}{d_{K,L}} \Big( \rho^{n+1}_{K,L}(\beta(s^{n+1}_{L})-\beta(s^{n+1}_{K}))\; p_K^{n+1}+\\
&\hspace{5cm}(\beta(s^{n+1}_{L})-\beta(s^{n+1}_{K}))\; g(p_K^{n+1})\Big),\\
 &E_{3}=\sn\Delta t \sum_{K \in \mathcal{T}} \sum_{L \in N(K)}
\Big(\rho_{K,L}^{n+1}G_1(s^{n+1}_K,s^{n+1}_{L};dp^{n+1}_{K,L})\; p_K^{n+1}+\\
&\hspace{5cm}G_2(s^{n+1}_K,s^{n+1}_{L};dp^{n+1}_{K,L})\; g(p_K^{n+1})\Big),\\
& E_{4}=\sn\Delta t \sum_{K \in \mathcal{T}}\Big( F^{(n+1)}_{1,K,L}\; p_K^{n+1}
+F^{(n+1)}_{2,K,L}\; g(p_K^{n+1})\Big),\\
 &E_{5}=\sn\Delta t \sum_{K \in \mathcal{T}}\abs{K} \Big( \rho(p^{n+1}_K)s^{n+1}_K f_{P,K}^{n+1}\; p_K^{n+1}
+(s^{n+1}_K-1) f_{P,K}^{n+1}\;g(p_K^{n+1})+ \\
&\hspace{3cm}f_{I,K}^{n+1}\; g(p_K^{n+1})\Big).
\end{align*}
To handle  the first  term of the equality (\ref{ch4:disc-est}), let us prove that : for all $s\ge 0$ and $s^\star\ge 0$,
\begin{equation}
\label{ch4:magic}
 \bigl(\rho(p)s-\rho(p^\star)s^\star\bigr)p+(s-s^\star)(\h(p)-\rho(p)p)\ge \h(p)s-\h(p^\star)s^\star.
\end{equation}
Indeed, denote $g(p) = \h(p)-\rho(p)p$ then $g'(p) = -\rho(p)$,
\begin{multline*}
\bigl(\rho(p)s-\rho(p^\star)s^\star\bigr)p+(s-s^\star)(\h(p)-\rho(p)p) \\
= s(\h(p)-s^\star(\rho(p^\star)p+ g(p))
= s\h(p) -s^\star\h(p^\star) + s^\star\bigl(\h(p^\star)-\rho(p^\star)p- g(p) \bigr).
\end{multline*}
We have to show that
$$
\h(p^\star)-\rho(p^\star)p- g(p)  \ge 0.
$$
We expand this quantity as follows,
$$
\h(p^\star)-\rho(p^\star)p- g(p) = g(p^\star)+\rho(p^\star)(p^\star-p)-g(p) =
g(p^\star)-g(p) -g'(p^\star)(p^\star-p),
$$
as the function $g$ is concave $(g^{''}(p)=-\rho'(p)\le 0)$ , we get
$$
g(p)\le g(p^\star)+g'(p^\star)(p-p^\star).
$$
So, (\ref{ch4:magic}) is established, and this yields to

\begin{equation}\label{ch4:est(11+12)}
 \sum_{K\in \mathcal{T}} \abs{K}s_K^N \h(p_K^N) -\sum_{K\in \mathcal{T}} \abs{K}s_K^0 \h(p_K^0)\le E_{1}
\end{equation}
Integrating  by parts, see {\it lemma} \ref{ch4:lemm-integ-parts}, we obtain
\begin{multline*}
  E_{2}= \frac{1}{2}\sn \Delta t \sum_{K \in \mathcal{T}}
\sum_{L \in N(K)}\frac{\abs{\sigma_{K,L}}}{d_{K,L}} \Big(\beta(s^{n+1}_{L})-\beta(s^{n+1}_{K})\Big)\\
 \Big( \rho^{n+1}_{K,L} (p_L^{n+1}-p_K^{n+1})+ (g(p_L^{n+1})-g(p_K^{n+1}))\Big).
\end{multline*}
Due to the correct choice of the density of the gas on each interface,
$$\rho^{n+1}_{K,L}=
\frac{(g(p_L^{n+1})-g(p_K^{n+1}))}{(p_K^{n+1}-p_L^{n+1})}$$
we succeed to obtain,
\begin{equation}\label{ch4:est(12+22)}
  E_{2}=0.
\end{equation}
The choice of the density on the interfaces is  the key point to vanish the dissipative term on saturation and obtain a uniform estimate on the discrete gradient of pressure $p$.\\
 Using the fact that the numerical fluxes $G_1$ and $G_2$ are conservative in the sense of (c) in (\ref{ch4:Hypfluxes}), we can apply {\it lemma} \ref{ch4:lemm-integ-parts} and we obtain
\begin{multline*} \label{ch4:}
  E_{3}=\frac{1}{2}\sn\Delta t \sum_{K \in \mathcal{T}} \sum_{L \in N(K)}
\rho_{K,L}^{n+1}\Big( G_2(s^{n+1}_K,s^{n+1}_{L};dp^{n+1}_{K,L})\\
-G_1(s^{n+1}_K,s^{n+1}_{L};dp^{n+1}_{K,L})\Big) \Big(p_L^{n+1}-p_K^{n+1}\Big),
\end{multline*}
Recall that inequality (\ref{ch4:tmob}),
$$\Big( G_2(a,b,c)-G_1(a,b,c)\Big) c=M(b){c^+}^2 + M(a){c^-}^2\geq m_0 c^2,$$
this with the hypothesis (H \ref{ch4:hyp:H6}) allow us to deduce that,
\begin{equation}\label{ch4:est(13+23)}
 m_0 \rho_m  \sum_{n=0}^{N-1}\Delta t\sum_{K\in \mathcal{T}}\sum_{L\in N(K)}
\frac{\abs{\sigma_{K,L}}}{d_{K,L}}
\abs{p_{K}^{n+1}-p_{L}^{n+1}}^2\le E_{3}.
\end{equation}
To handle the other terms of the equality (\ref{ch4:disc-est}), firstly let us remark that the numerical flux satisfies $F_{1,K,L}^{n+1}= -F_{1,L,K}^{n+1}$ and $F_{2,K,L}^{n+1}= -F_{2,L,K}^{n+1}$, so we integrate by parts and we obtain
$$
E_{4}= \frac1 2\sn\Delta t \sum_{K \in \mathcal{T}}\sum_{L \in N(K)} |\sigma_{K,L}|\Big (  F^{(n+1)}_{1,K,L} (p_K^{n+1} - p_L^{n+1} )
+F^{(n+1)}_{2,K,L} (g(p_K^{n+1}) - g(p_L^{n+1} ) \Big),
$$
use now the fact that the  mobilities and densities are bounded from (H\ref{ch4:hyp:H3})-(H\ref{ch4:hyp:H6}), and  the map $g$ is uniformly Lipschitz, we have, there exists a positive constant independent of $\Delta t$ and $h$ such that
$$
|E_{4}| \le C_1\sn\Delta t \sum_{K \in \mathcal{T}}\sum_{L \in N(K)} |\sigma_{K,L}| |p_K^{n+1} - p_L^{n+1}|.
$$
From the following inequality
$
|\sigma_{K,L}| = (|\sigma_{K,L}| d_{K,L})^{\frac{1}{2}}
\frac{ |\sigma_{K,L}|^{\frac{1}{2}}}{ d_{K,L} ^{\frac{1}{2}}}
$,
 and apply the Cauchy-Schwarz inequality to obtain
\begin{align*}
|E_{4}| \le &C_1\sn\Delta t \sum_{K \in \mathcal{T}}\sum_{L \in N(K)} |\sigma_{K,L}| d_{K,L}\\
&+\frac{m_0 \rho_m}{4}  \sum_{n=0}^{N-1}\Delta t\sum_{K\in \mathcal{T}}\sum_{L\in N(K)}
\frac{\abs{\sigma_{K,L}}}{d_{K,L}} \abs{p_{K}^{n+1}-p_{L}^{n+1}}^2\\
  \le & C_1 T |\Omega| + \frac{m_0 \rho_m}{4}  \sum_{n=0}^{N-1}\Delta t\sum_{K\in \mathcal{T}}\sum_{L\in N(K)}
\frac{\abs{\sigma_{K,L}}}{d_{K,L}} \abs{p_{K}^{n+1}-p_{L}^{n+1}}^2.
\end{align*}
The last term will be absorbed by the dissipative term on global pressure from the estimate \eqref{ch4:est(13+23)}.

In order to estimate $E_5$, using again the fact that the densities are bounded and the map $g$ is sublinear (i.e.$ |g(p)|\le C |p|$), we have
$$
 \abs{E_{5}} \leq
 C_1 \sum_{n=0}^{N-1}\Delta t\sum_{K\in \mathcal{T}} \abs{K} (f_{P,K}+ f_{I,K}^{n+1} ) \abs{p_K^{n+1}}
$$
we apply Holder inequality to deduce,
\begin{align*}
 \abs{E_{5}} \leq
  &C_1 \big(\sum_{n=0}^{N-1}\Delta t \sum_{K\in \mathcal{T}} \abs{K} \abs{f_{P,K}+ f_{I,K}^{n+1}}^2\big)^{\frac{1}{2}}  \big(\sum_{n=0}^{N-1}\Delta t\sum_{K\in \mathcal{T}} \abs{K} \abs{p_K^{n+1}}^2\big)^{\frac{1}{2}}\\
 \leq & C_1(\norm{f_P+f_I}_{L^2(Q_T)}\big( \sum_{n=0}^{N-1}\Delta t
\norm{p^{n+1}_h}_{L^2({\O})}^2\big)^{\frac{1}{2}}
 \end{align*}
Now, from the discrete Poincar\'e inegqality lemma  \ref{ch4:lem-disc-sob},  leads to,
$$
 \abs{E_{5}} \leq
 C_2 \big( \sum_{n=0}^{N-1}\Delta t
\norm{p^{n+1}_h}_{H_h}^2\big)^{\frac{1}{2}},
 $$
 where $C_2$ is a constant depends only on $\norm{f_P+f_I}_{L^2(Q_T)}$.
Finally, under the assumption (H\ref{ch4:hyp:H4})  on the source terms and as an application of Young's inequality ($a\cdot b \leq \eta a^2+ \frac{b^2}{4\eta}$.), we get
\begin{align}\label{ch4:est(4,5,6)}
 \abs{E_{5}} \leq C3 + \frac{m_0 \rho_m }{4} \sum_{n=0}^{N-1}\Delta t\sum_{K\in \mathcal{T}}\sum_{L\in N(K)}
\frac{\abs{\sigma_{K,L}}}{d_{K,L}}
\abs{p_{K}^{n+1}-p_{L}^{n+1}}^2
\end{align}
this estimate (\ref{ch4:est(4,5,6)}), with (\ref{ch4:est(11+12)}), (\ref{ch4:est(12+22)})  and (\ref{ch4:est(13+23)}) achieve the proof of (\ref{ch4:est:grad-norm}).

To prove the estimate \eqref{ch4:est:grad-norm1}, we multiply  the water discrete equation in \eqref{ch4:prob-discr2}   by $\beta(s^{n+1}_{K}) $  then summing
the resulting equation over $K$ and $n$, and this yields to

\begin{equation}\label{ch4:disc-est1}
  J_1+ J_2 + J_3  = 0
  \end{equation}
where
\begin{equation*}
  \begin{split}
& J_{1} = \sn \sum_{K \in \mathcal{T}}\abs{K}\phi_K(s^{n+1}_K-s^{n}_K)\cdot \beta(s_K^{n+1}),
\\
& J_{2} = -\sn \Delta t \sum_{K \in \mathcal{T}}
\sum_{L \in N(K)}\frac{\abs{\sigma_{K,L}}}{d_{K,L}}(\beta(s^{n+1}_{L})-\beta(s^{n+1}_{K}))\cdot \beta(s_K^{n+1}),
\\
& J_{3}=\sn\Delta t \sum_{K \in \mathcal{T}} \sum_{L \in N(K)}
G_2(s^{n+1}_K,s^{n+1}_{L};dp^{n+1}_{K,L})\cdot \beta(s_K^{n+1})\\
&~~~~+\sn\Delta t \sum_{K \in \mathcal{T}}F^{(n+1)}_{2,K,L}\cdot \beta(s_K^{n+1})
\\
&~~~~+\sn\Delta t \sum_{K \in \mathcal{T}} \abs{K}\Big((s^{n+1}_K-1) f_{P,K}^{n+1} + f_{I,K}^{n+1}\Big) \cdot \beta(s_K^{n+1}).
\end{split}
\end{equation*}

Let $\displaystyle B(s)=\int_0^s \beta(r)\,dr$. From the convexity of $B(s)$
(recall that $\beta''(s)=a(s)\ge 0$), we obtain
\begin{equation}\label{ch4:est:E1}
\begin{split}
J_1 &= \sn \sum_{K \in \mathcal{T}} \abs{K}(s^{n+1}_{K}-s^{n}_{K}) \beta(s^{n+1}_{K})\\
&\ge \sn \sum_{K \in \mathcal{T}} \abs{K}(B(s^{n+1}_{K})-{B}(s^{n}_{K}))\\
&=\sum_{K \in \mathcal{T}} \abs{K}
\left({B}(s^{N}_{K})-{B}(s^{0}_{K})\right)
\end{split}
\end{equation}

Applying {\it lemma} \ref{ch4:lemm-integ-parts}, we obtain
\begin{equation}\begin{split} \label{ch4:est:E2}
 J_{2} &= \frac{1}{2}\sn \Delta t \sum_{K \in \mathcal{T}}
\sum_{L \in N(K)}\frac{\abs{\sigma_{K,L}}}{d_{K,L}}(\beta(s^{n+1}_{L})-\beta(s^{n+1}_{K}))\cdot (\beta(s_L^{n+1})-\beta(s_K^{n+1}))
\\&=\frac{1}{2}\sn \Delta t \sum_{K \in \mathcal{T}}
\sum_{L \in N(K)}\frac{\abs{\sigma_{K,L}}}{d_{K,L}}\rho^{n+1}_{K,L}\abs{\beta(s^{n+1}_{L})-\beta(s^{n+1}_{K})}^2.
\end{split}
\end{equation}

The other terms in the equality (\ref{ch4:disc-est1}) can be treated as (\ref{ch4:est(4,5,6)}), using Holder's and Young's inequalities with the help of \textit{lemma} \ref{ch4:lem-disc-sob} and  the assumptions on mobilities (H\ref{ch4:hyp:H3}), source terms (H\ref{ch4:hyp:H4}) and densities (H\ref{ch4:hyp:H6}) to get,

\begin{equation}\begin{split} \label{ch4:est:E3456}
\abs{ J_3 }\leq C + \frac{1}{4} \sum_{n=0}^{N-1}\Delta t\sum_{K\in \mathcal{T}}\sum_{L\in N(K)}
\frac{\abs{\sigma_{K,L}}}{d_{K,L}}
\abs{\beta(s_{K}^{n+1})-\beta(s_{L}^{n+1})}^2
\end{split}
\end{equation}
Now, collecting (\ref{ch4:est:E1}), (\ref{ch4:est:E2}) and (\ref{ch4:est:E3456}) we obtain,
\begin{equation*}\label{ch4:}\begin{split}
&
\sum_{K\in \mathcal{T}} \abs{K}B(s_K^N)  -\sum_{K\in \mathcal{T}} \abs{K}B(s_K^0) \\
&
+\frac{1}{4} \sum_{n=0}^{N-1}\Delta t\sum_{K\in \mathcal{T}}\sum_{L\in N(K)}
\frac{\abs{\sigma_{K,L}}}{d_{K,L}}
\abs{\beta(s_{K}^{n+1})-\beta(s_{L}^{n+1})}^2 \le C,
\end{split}\end{equation*}
for some constant $C\geq0.$ This concludes the proof of proposition \ref{ch4:prop:LPBV}.
\end{proof}

\section{Existence of the finite volume scheme}\label{ch4:sec-exist-est}

The existence of a solution to the finite volume scheme
will be obtained with the help of the following lemma
proved in \cite{lionsj} and \cite{Tem:2001}.

\begin{lemma}\label{ch4:lem:exist-classic}
Let $\mathcal{A}$ be a finite dimensional Hilbert space with scalar product $[\cdot,\cdot]$
and norm $\norm{\cdot}$, and let $\mathcal{P}$ be a continuous mapping from $\mathcal{A}$
into itself such that
$$
[\mathcal{P}(\xi),\xi]>0 \text{ for } \norm{\xi} =r>0.
$$
Then there exists $\xi \in \mathcal{A}$ with $\norm{\xi} \le r$ such that
$$
\mathcal{P}(\xi)=0.
$$
\end{lemma}

The existence for the finite volume scheme is given in
\begin{proposition}
\label{ch4:prop:exist-fv}
Let $\mathcal{D}$ be an admissible discretization of $Q_{T}$.
Then the problem \eqref{ch4:prob-discr1}-\eqref{ch4:prob-discr2} admits at least one solution
$(p^{n}_{K},s^{n}_{K})_{(K,n) \in \Omega_R\times  \{0,\ldots,N\}}$.
\end{proposition}

\begin{proof}
At the beginning of the proof, we set the following notations;
\begin{align*}
\mathcal{M}:=Card(\mathcal{T})\\
s_{\mathcal{M}}:=\{s^{n+1}_K \}_{K\in\mathcal{T}} \in\R^{\mathcal{M}},\\
p_{\mathcal{M}}:=\{p^{n+1}_K \}_{K\in\mathcal{T}} \in\R^{\mathcal{M}}
\end{align*}
We define the map $\mathcal{T}_h:\R^\mathcal{M}\times\R^\mathcal{M}\longrightarrow \R^\mathcal{M}\times\R^\mathcal{M},$

$$\mathcal{T}_h(s_{\mathcal{M}}, p_{\mathcal{M}})=(\{\mathcal{T}_{1,K}\}_{K\in \mathcal{T}}, \{\mathcal{T}_{2,K}\}_{K\in \mathcal{T}})\,\, \text{where,} $$
\begin{align}
&\mathcal{T}_{1,K}=\abs{K}\phi_K\frac{\rho(p^{n+1}_K)s^{n+1}_K-\rho(p^{n}_K)s^{n}_K}{\Delta t}
-\sum_{L \in N(K) }
\frac{\abs{\sigma_{K,L}}}{d_{K,L}}\rho^{n+1}_{K,L}(\beta(s^{n+1}_{L})-\beta(s^{n+1}_{K}))
\notag\\
&\qquad \qquad + \sum_{L \in N(K) } \rho^{n+1}_{K,L}G_1(s^{n+1}_K,s^{n+1}_{L};dp^{n+1}_{K,L}) +F^{(n+1)}_{1,K,L}+ \abs{K}\rho(p^{n+1}_K)s^{n+1}_K f_{P,K}^{n+1},\label{ch4:def-T1} \\
&\mathcal{T}_{2,K}=\abs{K}\phi_K\frac{s^{n+1}_K-s^{n}_K}{\Delta t}- \sum_{L \in N(K) }
\frac{\abs{\sigma_{K,L}}}{d_{K,L}}(\beta(s^{n+1}_{L})-\beta(s^{n+1}_{K}))\notag\\
&\qquad \qquad + \sum_{L \in N(K) } G_2(s^{n+1}_K,s^{n+1}_{L};dp^{n+1}_{K,L}) + F^{(n+1)}_{2,K,L}+ \abs{K}(s^{n+1}_K-1) f_{P,K}^{n+1}
+\abs{K} f_{I,K}^{n+1}\label{ch4:def-T2}.
\end{align}
Note that $\mathcal{T}_h$ is well defined as a continues function. Also we define the following homeomorphism $\mathcal{F}:\R^\mathcal{M}\times\R^\mathcal{M}\mapsto \R^\mathcal{M}\times\R^\mathcal{M}$ such that,
$$\mathcal{F}(p_{\mathcal{M}},s_{\mathcal{M}})=(p_{\mathcal{M}}, v_{\mathcal{M}})$$
where, $v_{\mathcal{M}}=\{g(p^{n+1}_K)+\beta(s^{n+1}_K)\}_{K\in \mathcal{T}}.$\\
Now let us consider the following continues mapping $\mathcal{P}_h$ defined as
\begin{align*}
\mathcal{P}_h(p_{\mathcal{M}}, v_{\mathcal{M}})&=\mathcal{T}_h\circ\mathcal{F}^{-1}(p_{\mathcal{M}}, v_{\mathcal{M}})\\
&=\mathcal{T}_h(s_{\mathcal{M}}, p_{\mathcal{M}}).
\end{align*}
Our goal now is to show that,
\begin{align}
[\mathcal{P}_h(p_{\mathcal{M}}, v_{\mathcal{M}}),(p_{\mathcal{M}}, v_{\mathcal{M}})] > 0,~~~~
 \text{for}\,\, \norm{(p_{\mathcal{M}}, v_{\mathcal{M}})}_{\R^{2\mathcal{M}}}=r>0,\label{ch4:goal-exis}
\end{align}
and for a sufficiently large $r$.\\
We observe that
\begin{equation*}
\begin{split}
[\mathcal{P}_h(p_{\mathcal{M}}, v_{\mathcal{M}}),(p_{\mathcal{M}}, v_{\mathcal{M}})] \ge &\frac{1}{\Delta t}
\sum_{K\in \mathcal{T}} \abs{K}s_K^{n+1} H(p_K^{n+1}) -\frac{1}{\Delta t} \sum_{K\in \mathcal{T}} \abs{K}s_K^n H(p_K^n)\\
 +&\frac{1}{\Delta t} \sum_{K \in \mathcal{T}} \abs{K} {B}(s^{n+1}_{K})-\frac{1}{\Delta t} \sum_{K\in \mathcal{T}} {B}(s^{n}_{K})\\
 +&C(m_0,\rho_m) \norm{p^{n+1}_{h}}^2_{H_h(\Omega)}+1/2\norm{\beta(s^{n+1}_{h})}^2_{H_h(\Omega)}-C,
\end{split}
\end{equation*}
for some constants $C(m_0,\rho_m),C>0$. This implies that
\begin{equation}\label{ch4:eq:disc-exist:5}
\begin{split}
[\mathcal{P}_h(p_{\mathcal{M}}, v_{\mathcal{M}}),(p_{\mathcal{M}}, v_{\mathcal{M}})]  \ge & -\frac{1}{\Delta t} \sum_{K\in \mathcal{T}} \abs{K}s_K^n H(p_K^n)
 -\frac{1}{\Delta t} \sum_{K\in \mathcal{T}} {B}(s^{n}_{K})\\
& +C(\norm{p^{n+1}_{h}}^2_{H_h(\Omega)}+\norm{\beta(s^{n+1}_{h})}^2_{H_h(\Omega)})-C',
\end{split}
\end{equation}
for some constants $C,C'>0$.
Finally using the fact that $g$ is a Lipschitz function, then there exists a constant $C>0$ such that
\begin{align*}
\norm{(\{p^{n+1}_K\}_{K\in \mathcal{T}}, \{g(p^{n+1}_K)+\beta(s^{n+1}_K)\}_{K\in \mathcal{T}})}_{\R^{2\mathcal{M}}}
\le C(\norm{\beta(s^{n+1}_{h})}_{H_h(\Omega)}) +\norm{p^{n+1}_{h}}_{H_h(\Omega)}.
\end{align*}
Using this to deduce from \eqref{ch4:eq:disc-exist:5} that (\ref{ch4:goal-exis}) holds for $r$ large enough.
Hence, we obtain the existence of at least one solution to the scheme
\eqref{ch4:prob-discr1}-\eqref{ch4:prob-discr2}.
\end{proof}

\section{Space and time translation estimates}\label{ch4:sec-compactness}

In this section we derive estimates on differences of space and time translates of the
function $\phi_h \rho(p_h)s_hB(s_h)$ which imply that the sequence
$\phi_h \rho(p_h)s_hB(s_h)$ is relatively compact in $L^1(Q_T)$.

We replace the study of
discrete functions $U^h=\phi_h \rho(p_h)s_h B(s_h)$ (constant per cylinder $\QKn:=(t^n,t^{n+1})]\times K$) by the study
of functions $\bar U^h=\phi_h \rho(\bar p_h)\bar s_h B(\bar s_h)$ piecewise continuous in $t$ for all $x$,
constant in $x$ for all volume $\K$, defined as
$$\bar U^h(t,x)=\sum_{n=0}^{N_h}\sum_{\ptK\in\TT_h} \frac
1{\delt}\Bigl(\;(t-n\delt)U^{n+1}_\ptK
\,+\,((n+1)\delt-t)U^{n}_\ptK \; \Bigr)\;\charu_\QKn(t,x).$$
For a given discrete field $\vec{\mathcal{F}_h}:=\sum_{\sigma_{K,L}} \vec{\mathcal{F}}_{K,L} \charu_{T_{K,L}}$, its discrete divergence is defined as a discrete function with entries on each control volume $K;$\\
$$\div_K \vec{\mathcal{F}_h}:=\frac{1}{\abs{K}}\sum_{L\in N(K)} \sigma_{K,L}\vec{\mathcal{F}}_{K,L}\cdot \eta_{K,L} $$
Observe that we can write the discrete scheme \eqref{ch4:prob-discr1}-\eqref{ch4:prob-discr2} in the following from:
\begin{equation*}
\begin{split}
&\phi_h\frac{\rho(p^{n+1}_h)s^{n+1}_h-\rho(p^{n}_h)s^{n}_h}{\Delta t}=\div_\ptTau \Frond^{n+1}_{1,h}+f^{n+1}_{1,h}\\
&\phi_h\frac{s^{n+1}_h-s^{n}_h}{\Delta t}=\div_\ptTau \Frond^{n+1}_{2,h}+f^{n+1}_{2,h}.
\end{split}
\end{equation*}
where,
\begin{align*}
\Frond^{n+1}_{1,h}:=\sum_{\sigma_{K,L}}\Big( \frac{\rho^{n+1}_{K,L}}{d_{K,L}}(\beta(s^{n+1}_{L})-\beta(s^{n+1}_{K}))-
\frac{\rho^{n+1}_{K,L}}{\sigma_{K,L}} G_1(s^{n+1}_K,s^{n+1}_{L};dp^{n+1}_{K,L})\\
- \frac{1}{\sigma_{K,L}}(\rho^2(p^{n+1}_K) M_1(s^{n+1}_K){\bf g}_{K,L}
- \rho^2(p^{n+1}_L) M_1(s^{n+1}_L){\bf g}_{L,K})\Big)\cdot \eta_{K,L} \charu_{T_{K,L}}
\end{align*}
\begin{align*}
\Frond^{n+1}_{2,h}:=\sum_{\sigma_{K,L}}\Big( \frac{1}{d_{K,L}}(\beta(s^{n+1}_{L})-\beta(s^{n+1}_{K}))-
\frac{1}{\sigma_{K,L}} G_2(s^{n+1}_K,s^{n+1}_{L};dp^{n+1}_{K,L})\\
- \frac{1}{\sigma_{K,L}}(\rho_2 M_2(s^{n+1}_L){\bf g}_{K,L}
- \rho_2 M_2(s^{n+1}_K){\bf g}_{L,K})\Big)\cdot \eta_{K,L} \charu_{T_{K,L}}
\end{align*}
\begin{equation*}
\begin{split}
&f^{n+1}_{1,h}:=\rho(p^{n+1}_K)s^{n+1}_K f_{P,K}^{n+1}     \\
&f^{n+1}_{2,h}:=-(s^{n+1}_K-1) f_{P,K}^{n+1}-\abs{K} f_{I,K}^{n+1}.
\end{split}
\end{equation*}


We also extend $\bar U^h$ by the constant in time value $U^{\ptTau,N_h\!+\!1}$
on $[\delt (N_h\!+\!1),+\infty)$; as to $\Frond_{1,h}$,  $\Frond_{2,h}$, $f_{1,h}$ and $f_{2,h}$, they are
  extended by zero values for $t>\delt (N_h\!+\!1)$.  The above
  definitions permit us to rewrite the equations \eqref{ch4:prob-discr1}-\eqref{ch4:prob-discr2} under
  the form
  \begin{equation}\label{ch4:eq:Kr-Scheme-bis}
  \begin{split}
&\partial_t \phi_h \rho(\bar p_h)\bar s_h =\div_\ptTau \Frond_{1,h}+f_{1,h}\\
&\partial_t  \phi_h \bar s_h=\div_\ptTau \Frond_{2,h}+f_{2,h},
\end{split}
\end{equation}
where $\Frond^{}_{1,h}$, $\Frond^{}_{2,h}$, $f^{}_{1,h} $ \text{ and } $f^{}_{2,h}$  are respectively the discrete functions of values $\Frond^{n+1}_{1,h}$, $\Frond^{n+1}_{2,h}$, $f^{n+1}_{1,h}$ \text{ and } $f^{n+1}_{2,h}$ on each interval $]t^n, t^{n+1}]$.

These equations are satisfied in $W^{1,1}(\R^+)$ in time, for a.e. $x\in \Omega$.

\begin{lemma}\label{ch4:sp-ti-trslt}
There exists positive a constant $C>0$ depending on
$\Omega$, $T$, $u_{0}$ and $v_{0}$ such that
\begin{equation}\label{ch4:sp-trslt}
\iint_{\Omega ' \times (0,T)}\abs{\bar U(t,x+y)-\bar U(t,x)}^2 \,dx \, dt
\le C\abs{y}(\abs{y}+2h),
\end{equation}
for all $y \in \R^3$ with $\Omega '=\{x \in \Omega, \, [x,x+y]\subset \Omega\}$, and
\begin{equation}\label{ch4:ti-trslt}
\iint_{\Omega \times
(0,T-\tau)}\abs{\bar U(t+\tau,x)-\bar U(t,x)}^2\,dx \,dt \le C(\tau+\Delta t),
\end{equation}
for all $\tau\in (0,T)$.
\end{lemma}

\begin{proof}

The proof is similar to that found in, e.g, \cite{Eymard_etal_II:2000}.

{\it Proof of \eqref{ch4:sp-trslt}}.
First to simplify the notation, we write
$$
\sum_{\sigma_{K,L}}\quad \text{instead of} \quad
\sum_{\{(K,L)\in {\mathcal{T}}^2,\,K\ne L,\,m(\sigma_{K,L})\ne 0\}}.
$$
Let $y \in \R^3$, $x \in \Omega '$, and $L \in N(K)$. We set
$$
\beta_{\sigma_{K,L}}=
\begin{cases}
    1, & \text{if the line segment $[x,x+y]$ intersects $\sigma_{K,L}$, $K$ and $L$},\\
    0, & \text{otherwise}.
\end{cases}
$$
Next, the value $c_{\sigma_{K,L}}$ is defined by
$\displaystyle c_{\sigma_{K,L}}=\frac{y}{\abs{y}}\cdot \eta_{K,L}$
with $c_{\sigma_{K,L}}>0$.
We observe that (see for more details \cite{EyGaHe:book})
\begin{equation}\label{ch4:est:space}\begin{split}
&\int_{\Omega '}\beta_{\sigma_{K,L}}(x) \,dx \le m(\sigma_{K,L}) \abs{y} c_{\sigma_{K,L}},\\
&\sum_{\sigma_{K,L}} \beta_{\sigma_{K,L}}(x)
c_{\sigma_{K,L}}d_{K,L}\le \abs{y}+2h.
\end{split}\end{equation}
With this and an application of the Cauchy-Schwarz inequality leads to
\begin{equation}\label{ch4:est:space:4}
      \begin{split}
&\iint_{(0,T)\times\Omega '}\abs{U^h(t,x+y)-U^h(t,x)}^2\,dx\\
&\qquad \le \sum_{\sigma_{K,L}} \beta_{\sigma_{K,L}}(x)c_{\sigma_{K,L}}
d_{K,L} \sn \Delta t \sum_{\sigma_{K,L}} \frac{\abs{U_L^{n+1}-U_K^{n+1}}^2}
{c_{\sigma_{K,L}}d_{K,L}} \int_{\Omega '} \beta_{\sigma_{K,L}}(x)\,dx\\
&\qquad \le  (\abs{y}+2h)\sn \Delta t
\sum_{\sigma_{K,L}}\frac{\abs{U_L^{n+1}-U_K^{n+1}}^2}
{c_{\sigma_{K,L}}d_{K,L}}\int_{\Omega '} \beta_{\sigma_{K,L}}(x)\,dx\\
&\qquad \le \abs{y} (\abs{y}+2h)\sn \Delta t
\sum_{\sigma_{K,L}}\frac{\abs{\phi \rho(p_L^{n+1})s_L^{n+1} B(s_L^{n+1})-\phi \rho(p_K^{n+1})s_K^{n+1} B(s_K^{n+1})}^2}
{c_{\sigma_{K,L}}d_{K,L}}\\
&\qquad \le C \abs{y} (\abs{y}+2h)\sn \Delta t
\sum_{\sigma_{K,L}}\frac{\Biggl(\abs{p_L^{n+1}-p_L^{n+1}}^2+\abs{\beta(s_L^{n+1})-\beta(s_K^{n+1})}^2\Biggl)}
{c_{\sigma_{K,L}}d_{K,L}},
\end{split}
\end{equation}
for some constant $C>0$.
In addition, we have
\begin{align*} \int_0^{+\infty}\!\!\!\!\int_{\Omega '}\!\! |\bar U^h(t,x\!+\!\dx)
\!-\! \bar U^h(t,x) |\,dxdt \leq &2 \int_0^{T}\!\!\!\!\int_{\Omega '}
\!\!|U^h(t,x\!+\!\dx) \!-\! U^h(t,x) |\,dxdt \\
+&2\delt_h\int_{\Omega '_\Delt}\!\!\!\! |U^h_0(x)|\,dx,
\end{align*}
where $U_0=\rho(p_0)s_0 B(s_0)$ and $\Omega '_\Delt=\{x \in \Omega, \, \dist(x,\Omega ')< \abs{\Delta}\}$. By (\ref{ch4:sp-trslt}), the assumption $\delt_h\to 0$ as $h\to 0$
and the boundedness of $(u^h_0)_h$ in $L^1(\Omega '_\Delt)$, then the space
translates of $\bar U^h$ on $\Omega '$ are estimated uniformly for all
sequence $(h_i)_i$ convergent to zero. In the sequel, we drop the subscript $i$ in the notation.

{\it Proof of \eqref{ch4:ti-trslt}}.  Now we show a
uniform estimate of the time translates of $(\bar U^h)_h$ :
\begin{equation}\label{ch4:eq:timetransl}
 \text{for all $\Delt\in (0,\tau]$}, \quad
{\displaystyle\int_0^{+\infty}\!\!\!\int_{\Omega}} |U^h(t\!+\!\Delt, x)\!
-\! U^h(t,x) |\,dx dt \leq \widetilde{\omega}(\tau)
\end{equation}
uniformly in $h$. Here $\widetilde{\omega}:\RR^+\longrightarrow
 \RR^+$ is a modulus of continuity, i.e., $\lim\limits_{\tau\to 0}
 \widetilde{\omega}(\tau)=0$.

\medskip
Let us construct $\widetilde{\omega}(\cdot)$ verifying
(\ref{ch4:eq:timetransl}). First fix $h$ and fix $\Delt \in(0,\tau]$.
  Denote by $I^h(\Delt)$ the left-hand side of
  \eqref{ch4:eq:timetransl}. For $t\geq 0$, set $W^h(t,\cdot)=\bar
  U^h(t+\Delt,\cdot)-\bar U^h(t,\cdot)$. Notice that
  $W^h(t,\cdot)\equiv 0$ for large $t$.

Take a standard family $(\rho_\delta)_\delta$ of mollifiers on $\R^l$
defined as $\rho_\delta(x):={\delta^{-l}}\rho(x/\delta)$, where $\rho$
is a Lipschitz continuous, nonnegative function supported in the unit
ball of $\R^l$, and $\int_{\R^l} \rho(x)\,dx=1$. In particular, we
have
$$|\grad \rho_\delta|\leq \frac C{\delta^{l+1}}.$$ Here and throughout
the proof, $C$ will denote a generic constant independent of $h$ and
$\delta$. For all $t>0$, define the function
$\ph(t,\cdot):\R^l\longrightarrow \R$ by $\ph(t):=\rho_\delta*(\sign
w^h(t)\charu_{\Omega '})$. In order to lighten the notation, we do not
stress the dependence of $\ph$ on $h$ and $\delta$. Discretize
$\ph(t,\cdot)$ on the mesh $\Tau_h$ by setting $\ph_\ptK(t)=\frac 1\mK
\int_\ptK \ph(t,x)\,dx$; we denote $\ph^h(t)$ the corresponding
discrete function. Denote $\size(\Tau_h):=\max_{\ptK\in\TT_h}
\diam(\K)$.  By the definition of $\ph(t,\cdot)$, the discrete
function $\ph^h(t)$ is null on the set $\Bigl\{ x\in\Omega \,\Bigl|\,
\dist(x,\overline{\Omega '})\geq \delta+\size(\Tau_h) \Bigr\}$, for all
$t$. Thus for all sufficiently small $h$ and $\delta$, the support of
$\ph^h(t)$ is included in some domain $\Omega ''$, $\Omega ''\subset\Omega$.

Note that
$$
\partial_t U^h=\partial_t (\rho(p_h)s_h B(s_h))=\partial_t (\rho(p_h)s_h) B(s_h)+\rho(p_h)s_h\partial_t B(s_h).
$$
Now for for all $x\in\K$, we multiply equation
\eqref{ch4:eq:Kr-Scheme-bis} by $\mK\ph(t)_\ptK$, integrate in $t$ on
$[s,s+\Delt]$, and make the summation over all $\K$. Finally, we
integrate the obtained equality in $s$ over $\R^+$ to get
\begin{equation}\label{ch4:eq:Estimate-with-wt}\begin{split}
&\int_0^{+\infty}\!\sum_\K \mK
\ph_\ptK(s)W_\ptK(s)\,ds=\\
& \int_0^{+\infty}\!\!\int_{s}^{s+\Delt}\sum_\K
\mK\,\ph_\ptK(t)B(s_K(t))\,\biggl(\,\div_\ptK [\Frond^{h}_1(t)] + (f_1^h(t))_\ptK
\biggr)\,dtds\\
+&\int_0^{+\infty}\!\!\int_{s}^{s+\Delt}\sum_\K
\mK\,\ph_\ptK(t)\rho(p_k(t))s_k(t)\beta(s_K(t))\,\biggl(\,\div_\ptK [\Frond^{h}_2(t)] + (f_2^h(t))_\ptK
\biggr)\,dtds.
\end{split}\end{equation}
Denote by $I^h_\delta(\Delt)$ the left-hand side of
\eqref{ch4:eq:Estimate-with-wt}. Denote
$Q''=(0,(N_h+1)\delt)\times\Omega ''$. Using hypothesis the definitions of discrete norms
and the Fubini theorem, we get
\begin{multline*}I^h_\delta(\Delt)\leq C\Delt \;\biggl(\;  \Bigl\|\,
\Frond^{h}\,\Bigr\|_{L^2(Q'')}\,(\max_{t>0}\max_\KIL
|\grad_\ptKIL \ph^h(t)| + \\
|\grad_\ptKIL B(s_h(t))|^2)\;+ \;
\|\ph^h\|_{L^\infty(Q'')}\|f^h\|_{L^1(Q'')}\;\biggr).
\end{multline*}
Now the  the $L^2_{loc}([0,T]\times\Omega)$ bounds on
$(\Frond^h)_h$,$(f^h)_h$, the bounds $|\ph(t,\cdot)|\leq 1$ and
$|\grad\ph(t,\cdot)|\leq C/{\delta^{l+1}}$,  yield the estimate
\begin{equation}\label{ch4:eq:Estimate-with-wt-2}
I^h_\delta(\Delt)\leq C\Delt (1+{\delta^{-l-1}})
\end{equation}
for all $h$ and $\delta$ small enough, uniformly in $h$. Now,
notice that by the definition of $\ph_\ptK(t)$,
\begin{equation}\begin{split}
\mK\,\Bigl(\,|W_\ptK(t)|\!-\!W_\ptK(t)\ph_\ptK(t)\,\Bigr)&=\dsp \mK\,
|W^h(t,x)|-W_\ptK(t)\!\! \int_\ptK \!\!\ph(t,x)\,dx\\
&=\! \int_\ptK\!
\Bigl(\,|W^h(t,x)|\!-\!W^h(t,x)\ph(t,x)\,\Bigr)\,dx;
\end{split}\end{equation}
 therefore
\begin{equation}\label{ch4:eq:I-I-equality}
I^h(\Delt)-I^h_\delta(\Delt)={\displaystyle\int_0^{+\infty}\!\!\!\int_{\Omega}}
\Bigl(|W^h(t,x)|-W^h(t,x)\ph(t,x)\,\Bigr)\,dxdt.
\end{equation}
Starting from this point, the argument of Kruzhkov
Ref.~\cite{Kruzhkov-zametki} applies exactly as for the ``continuous'' case.  Set $U'_\delta:=\Bigl\{x\in
\R^l\,\Bigr|\,\dist(x,\ptl{\Omega '})<\delta\Bigr\}$; notice that
$U'_\delta\subset\Omega ''\subset\Omega$ for all $\delta$ small enough.
Notice that without loss of restriction, the boundary of $\Omega '$ can be
chosen regular enough so that to ensure that $\meas(U'_\delta)$ goes
to zero as $\delta\to 0$. By the result of Step 1 of the lemma and the
Frechet-Kolmogorov theorem, the family
$\Bigl(\displaystyle\int_0^{+\infty} \!\!\!|W^h(t,\cdot)|dt\Bigr)_{h}$
is relatively compact in $L^1_{loc}({\Omega})$. Therefore these functions
are equi-integrable on ${\Omega ''}$, so that
${\displaystyle\int_0^{+\infty}\!\!\!\int_{U'_\delta}} |W^h(t,x)|dxdt
\leq \hat{\omega}(\delta)$ uniformly in $h$, with $\lim_{\delta \to 0}
\hat{\omega}(\delta)=0$. Now by the definition of $\ph$, from formula
\eqref{ch4:eq:I-I-equality} we deduce that
\begin{align*} |I^h(\Delt)\!-\!I^h_\delta(\Delt)|\leq 2
&\int_0^{+\infty}\!\!\!\!\int_{U'_\delta}\!\!|W^h(t,x)| dxdt\\
+&\int_0^{+\infty}\!\!\!\!\int_{\Omega '\setminus U'_\delta}
\biggl||W^h(t,x)| -W^h(t,x) (\rho_\delta*{\rm sign}\,
W^h(t))(x)\biggr|\, dxdt,
\end{align*}
 the first term in the right-hand side accounts for the action of
the truncation $\charu_{\Omega '}$ in the definition of $\ph$. Using the
standard properties of $\rho_\delta$, we infer
\begin{multline*}
|I^h(\Delt)\!-\!I^h_\delta(\Delt)|\leq 2\hat{\omega}(\delta)\\
+\int_0^{+\infty}\!\!\!\int_{\Omega '\setminus
  U'_\delta}\int_{\R^l}\!  \rho_\delta(x\!-\!y)
\biggl||W^h(t,x)|-W^h(t,x) \,{\rm sign}\, W^h(t,y)\biggr|\, dydxdt.
\end{multline*}
 Now note the key inequality:
$$ \text{ $\forall a,b \in \RR$ \quad $\Bigl||a|-a\, {\rm sign}\,
  b\Bigr|\leq 2\,|a-b|$}.
$$ Setting $\sigma:=(x\!-\!y)/\delta$, we
infer \begin{equation}\label{ch4:eq:I-I-inequality-bis}
\begin{array}{l}
|I^h(\Delt)\!-\!I^h_\delta(\Delt)|\,\leq \, 2\hat{\omega}(\delta)
\,+\, 2\displaystyle\int_0^{+\infty}\!\!\!\int_{\Omega '}\int_{\R^l}\!
\rho_\delta(x\!-\!y)|W^h(t,x)\!-\!W^h(t,y)|\, dydxdt \,\leq\, \\[10pt]
\quad \leq\, 2\hat{\omega}(\delta) \,+\,
2\displaystyle\int_{\R^l}\!\rho(\sigma)\int_0^{+\infty}\!\!\!\int_{\Omega '}
|\bar u^h(t,x)\!-\!\bar u^h(t,x\!-\!\delta\sigma)|\, dxdt\,d\sigma
\,\leq\, 2\hat{\omega}(\delta) + 2\omega(\delta),
\end{array}
\end{equation}
where $\omega(\cdot)$ is the modulus of continuity controlling the
space translates of $\bar u^h$ in $\Omega'$. Indeed, by Steps 1 and 2
of the proof, one can choose $\omega(\cdot)$ independent of
$h$. Combining \eqref{ch4:eq:Estimate-with-wt-2} with
\eqref{ch4:eq:I-I-inequality-bis}, we conclude that the function
$$\widetilde{\omega}(\tau):=\inf\limits_{\delta>0} C
\left\{\;\tau\,(1+{\delta^{-l-1}}) + 2\hat{\omega}(\delta) +
2\omega(\delta) \;\right\}$$ upper bounds the quantity $I^h$. Because
$\widetilde{\omega}(\tau)$ tends to $0$ as $\tau\to 0$, this proves
\eqref{ch4:eq:timetransl}.
\end{proof}

\section{Convergence of the finite volume scheme}\label{ch4:sec:conv}
\begin{proposition} \label{ch4:propconvh}
There exists a subsequences, still denoted $(U_{h},s_h)_h$, such that, as $h\to 0$
\begin{align}
\label{ch4:cvhuubar}
&\Vert U^h-\bar U^h\Vert_{L^1(\Omega ')}\longrightarrow 0,\\
\label{ch4:cvhp}
&U_{h}\longrightarrow U \text{ strongly in }L^p(Q_T) \text{ and a.e. in } Q_T \text{ for all } p\geq1,\\
\label{ch4:cvhsw}
&s_{h}\longrightarrow s \text{ strongly in } L^p(Q_T)  \text{ for all } p>1,\\
\label{ch4:cvhpw}
&\Grad_h \beta(s_{h})\longrightarrow \Grad \beta(s) \text{ weakly in } (L^2(Q_T))^3,\\
\label{ch4:cvhs}
&\Grad_h p_{h}\longrightarrow \Grad p \text{ weakly in } (L^2(Q_T))^3,\\
\label{ch4:cvhsb}
& p_{h}\longrightarrow  p \text{ weakly in } L^2(Q_T).\\
\end{align}
Furthermore,
\begin{align}
\label{ch4:cvpps}
&s_{h}\longrightarrow s \text{ a.e. in } Q_T, \text{ and } 0\le s \le 1 \text{ a.e. in } Q_T,\\
&U=\phi \rho(p)s B(s) \text { a.e. in } Q_T \label{ch4:r=rhos}
\end{align}
Finally, we have,
\begin{align}
\label{ch4:f1f2}
&f_1(p_{h})f_2(s_{h})\longrightarrow f_1(p)f_2(s) \text { a.e. in } Q_T, \forall f_1,f_2 \in {{\mathcal C}}_b^0(\R) \text { such that } f_2(0)=0.
\end{align}
\end{proposition}
\begin{proof}
For the first convergence (\ref{ch4:cvhuubar}) it is useful to introduce the following inequality, for all $a,b\in\R$,
 $$\int_0^1 |\theta
 a + (1-\theta)b|\,d\theta\geq \frac{1}{2}(|a|+|b|)$$
 Applying this inequality to $a=u_\ptTau^{(n+1)}-u_\ptTau^{n}$,
 $b=u_\ptTau^{n}-u_\ptTau^{(n-1)}$, from the definition of $\bar U^h$
 we deduce
$$ \int_0^T\int_{\Omega '}|U^h(t,x)-\bar U^h(t,x)|\,dxdt \leq
 2\,\int_0^{T\!+\delt_h}\int_{\Omega '} |\bar U^h(t\!+\!\delt_h,x)-\bar
 U^h(t,x)|\,dxdt.
$$ Since $\delt_h$ tends to zero as $h\to 0$, estimate
 \eqref{ch4:eq:timetransl} in Lemma \ref{ch4:sp-ti-trslt} implies that the right-hand side of
 the above inequality converges to zero as $h$ tends to zero, and this established (\ref{ch4:cvhuubar}).

By the Riesz-Frechet-Kolmogorov compactness criterion, the relative compactness of $(\bar U^h)_h$ in
$L^1(Q_T)$ is a consequence of Lemma \ref{ch4:sp-ti-trslt}. Now, the convergence  (\ref{ch4:cvhp}) in $L^1(Q_T)$ and a.e in  $Q_T$ becomes a consequence of (\ref{ch4:cvhuubar}). Due to the fact that $U^h$ is bounded, we establish the convergence in $L^1(Q_T)$.

In order to prove the third convergence (\ref{ch4:cvhsw}), we reproduce the previous lemma \ref{ch4:sp-ti-trslt} for $U^h=s_h B(s_h)$,  and as an application of
the Riesz-Frechet-Kolmogorov compactness criterion  we establish (\ref{ch4:cvhsw}).

For the weak convergence of the discrete gradient  of the global pressure, let us recall the piecewise approximation $\nabla_h p_h$
of $\nabla p_h$ in $Q_t$:
$$
\Grad_h p_{h}(t,x)=\begin{cases}
l \frac{P^n_{L}-P^n_{K}}{d_{K,L}}\eta_{K,L}&\text{ if $(t,x) \in (t^n,t^{n+1})\times T_{K,L}$},\\
0&\text{ if $(t,x) \in (t^n,t^{n+1})\times T^{\text{ext}}_{K,\sigma}$},
\end{cases}
$$
for all $K \in \mathcal{T}$ and $0\leq n\leq N_h$.
It follows from proposition \ref{ch4:prop:LPBV} that, the sequence $(\Grad_h p_{h})_h$ is bounded in $(L^2(Q_T))^3$, and as a consequence of the discrete Poincaré  inequality, the sequence $(p_h)_h$ is bounded in $L^2(Q_T)$. Therefore there exist two functions $p\in L^2(Q_T)$ and $\zeta \in (L^2(Q_T))^3 $ such that (\ref{ch4:cvhsb}) holds and
$$\Grad_h p_{h}\longrightarrow \zeta \text{ weakly in } (L^2(Q_T))^3.$$
It remains to identify $\nabla p$ by $\zeta$ in the sense of distributions. For that, it is enough to show as $h\to 0$:
$$ E_h:=\int \int_{Q_T}\nabla_h p_h \cdot\varphi \,dxdt+\int \int_{Q_T} p_h \cn\varphi \,dxdt \longrightarrow 0,~~~~\forall \varphi\in \mathcal{D}(Q_T)^3.$$
Let $h$ be small enough such that $\varphi$ vanishes in $T^{\text{ext}}_{K,\sigma}$ for all $K\in\mathcal{T}$. In view of $\eta_{K,L}=-\eta{L,K}$ we obtain for all $t\in (t^n,t^{n+1})$
\begin{align*}
 \int_{\O} p_h \cn \varphi(t,x) \,dx&=\sum_{K\in \mathcal{T}} \int_K p_h \cn \varphi(t,x)\,dx\\
&=\sum_{K\in \mathcal{T}}\sum_{L\in N(K)} p^n_K\int_{\sigma_{K,L}}\varphi (t,s)\cdot \eta_{K,L}\,ds\\
&=\frac{1}{2}\sum_{K\in \mathcal{T}}\sum_{L\in N(K)}(p^n_K-p^n_L)\int_{\sigma_{K,L}}\varphi (t,s)\cdot \eta_{K,L}\,ds.
\end{align*}
Now, from the definition of the discrete gradient,
\begin{align*}
 \int_{\O} \nabla_h p_h  \varphi(t,x) \,dx
&=\frac{1}{2}\sum_{K\in \mathcal{T}}\sum_{L\in N(K)}\int_{T_{K,L}}\nabla_h p_h  \varphi(t,x) \,dx\\
&=\frac{1}{2}\sum_{K\in \mathcal{T}}\sum_{L\in N(K)}\frac{l}{d_{K,L}}(p^n_L-p^n_K) \int_{T_{K,L}} \varphi(t,x)\cdot \eta_{K,L} \,dx
\end{align*}
Then,
\begin{equation*}
\begin{split}
E_h= \frac{1}{2}\sum_{K\in \mathcal{T}}\sum_{L\in N(K)}\sigma_{K,L}(p^n_L-p^n_K) \Big(\frac{1}{\sigma_{K,L}} \int_{\sigma_{K,L}}\varphi (t,s)\cdot \eta_{K,L}\,ds \\- \frac{1}{T_{K,L}}\int_{T_{K,L}} \varphi(t,x)\cdot \eta_{K,L} \,dx \Big)
\end{split}
\end{equation*}
Due to the smoothness of $\varphi$, one gets
$$ \Big| \frac{1}{\sigma_{K,L}} \int_{\sigma_{K,L}}\varphi (t,s)\cdot \eta_{K,L}\,ds - \frac{1}{\abs{T_{K,L}}}\int_{T_{K,L}} \varphi(t,x)\cdot \eta_{K,L} \,dx\Big| \leq C\; h,
$$
and the Cauchy-Scharwz inequality with proposition \ref{ch4:prop:LPBV}
\begin{align*}
|E_h|&\leq Ch \sum_{n=0}^{N-1}\Delta t \sum_{K\in \mathcal{T}}\sum_{L\in N(K)} |\sigma_{K,L}||p^n_K-p^n_L|\\
&\leq Ch \sum_{n=0}^{N-1}\Delta t \sum_{K\in \mathcal{T}}\sum_{L\in N(K)} |\sigma_{K,L}|d_{K,L}\\
&\leq Ch |\O| T.
\end{align*}
Now, for the identification of the limit (\ref{ch4:r=rhos}):\\
Due to the monotonicity of the function $\rho$, we have
$$\int_{Q_T}(\phi_h \rho(p_h) s_h B(s_h)-\phi_h \rho(v) s_h B(s_h))\,dxdt \geq 0,~~ \forall v\in L^2(Q_T),$$
this with the strong convergence (\ref{ch4:cvhp}) and the weak convergence (\ref{ch4:cvhsb}) lead to,
$$\int_{Q_T}(U-\phi \rho(v) s B(s))\,dxdt \geq 0,~~ \forall v\in L^2(Q_T).$$
Finally, choose $v=p+\delta w$ with $\delta \in ]0,1]$ and $w\in  L^2(Q_T)$, then
$$\int_{Q_T}(U-\phi \rho(p+\delta w) s B(s))w\,dxdt \geq 0$$
letting $\delta$ goes to zero, we establish (\ref{ch4:r=rhos}).\\
To conclude the a.e. convergence (\ref{ch4:f1f2}),on one hand, when $s_{h}\to s=0$ a.e., $f_1(p_{h})f_2(s_{h})\to 0=f_1(p)f_2(s)$ a.e.
(since $f_2(0)=0$ and $f_1$ is bounded). On the other hand, when $s_{h}\to s\neq 0$, in light of (\ref{ch4:cvhp}) we have $f_1(p_{h})\to f_1(p_1)$ a.e.. Then, $f_1(p_{h})f_2(s_{h})\to f_1(p)f_2(s)$ since $f_1$, $f_2$ are  continuous and this establish (\ref{ch4:f1f2}).\end{proof}

\begin{theorem}
Assume (H\ref{ch4:hyp:H1})-(H\ref{ch4:hyp:H6}) hold. Then the functions p,s defined in proposition \ref{ch4:propconvh} constitute a weak solution of the system (\ref{ch4:gas})-(\ref{ch4:water}).
\end{theorem}
\begin{proof}
Let $T$ be a fixed positive constant and $\varphi \in \mathcal{D}([0,T)\times \overline{\Omega})$.

$\bullet$ Convergence of the discrete water equation\\
For the discrete water equation, we multiply the equation  \eqref{ch4:prob-discr2} by $\Delta t\varphi(t^{n+1},x_K)$ for all
$K \in \TT$ and $n \in  \{0,\ldots,N\}$. This yields

\begin{equation*}
\mathfrak{C}_1^{h}+\mathfrak{C}_2^{h}+\mathfrak{C}_3^{h}+\mathfrak{C}_4^{h}+\mathfrak{C}_5^{h}+\mathfrak{C}_6^{h}=0
\end{equation*}
where
\begin{equation*}\begin{split}
\mathfrak{C}_1^h&=\sn \sum_{K\in \TT}\abs{K}\phi_K(s^{n+1}_K-s^{n}_K)\varphi(t^{n+1},x_K),\\
\mathfrak{C}_2^h&=-\sn \Delta t \,
\sum_{K\in \TT} \sum_{L\in N(K)}
\frac{\abs{\sigma_{K,L}}}{d_{K,L}}
(\beta(s^{n+1}_{L})-\beta(s^{n+1}_{K}))\varphi(t^{n+1},x_K),\\
\mathfrak{C}_3^h&=\sn \Delta t \sum_{K\in \TT}\sum_{L\in N(K)}
 G_2(s^{n+1}_K,s^{n+1}_{L};dp^{n+1}_{K,L})\varphi(t^{n+1},x_K),\\
\mathfrak{C}_4^h&=\sn \Delta t \sum_{K\in \TT}\Big( \rho_2 M_2(s^{n+1}_L)\sum_{L \in N(K) }\int_{K/L}({\bf g}\cdot \eta_{K,L})^+\, d\gamma(x)\\
&- \rho_2 M_2(s^{n+1}_K)\sum_{L \in N(K) }\int_{K/L}({\bf g}\cdot \eta_{K,L})^-\, d\gamma(x)   \Big)\varphi(t^{n+1},x_K),\\
\mathfrak{C}_5^h&=\sn \Delta t \sum_{K\in \TT}\abs{K}(s^{n+1}_K-1) f_{P,K}^{n+1}\varphi(t^{n+1},x_K)\\
\mathfrak{C}_6^h&=\sn \Delta t \sum_{K\in \TT}\abs{K} f_{I,K}^{n+1}\varphi(t^{n+1},x_K).
\end{split}\end{equation*}

Performing integration by parts and keeping in mind
that $\varphi(T,x_K)=0$ for all $K \in \TT$, we obtain
\begin{equation*}
\begin{split}
\mathfrak{C}_1^h&=-\sn \sum_{K\in \TT}\abs{K}\phi_K s^{n+1}_K(\varphi(t^{n+1},x_K)-\varphi(t^n,x_K))
-\sum_{K\in \TT}\abs{K}\phi_K s^{0}_K\varphi(0,x_K)\\
&=-\sn \sum_{K\in \TT} \int_{t^n}^{t^{n+1}}\int_K \phi_K s^{n}_K
\partial_t \varphi(t,x_K)\dx \dt-\sum_{K\in \TT}\int_K \phi_K s_0(x) \varphi(0,x_K)\dx\\
&=:-\mathfrak{C}^h_{1,1}-\mathfrak{C}^h_{1,2}.
\end{split}
\end{equation*}
Let us also introduce
\begin{equation*}
\begin{split}
\mathfrak{C}^{h,*}_1 &=-\sn \sum_{K\in \TT} \int_{t^n}^{t^{n+1}}\int_K \phi_K s^{n}_K
\partial_t \varphi(t,x)\dx \dt-\int_\Omega \phi_h s_0(x)\varphi(0,x)\dx\\
&=:-\mathfrak{C}_{1,1}^{h,*}-\mathfrak{C}_{1,2}^{h,*}.
\end{split}
\end{equation*}
Now and due to the fact that the saturation and the porosity functions are bounded, we have
\begin{equation*}\begin{split}
\abs{\mathfrak{C}_{1,1}^{h}-\mathfrak{C}_{1,1}^{h,*}}
&= \abs{\sn \sum_{K\in \TT}\phi_K s^{n}_{K} \int_{t^n}^{t^{n+1}}\int_K
\Bigl( \partial_t \varphi(t,x_K)-\partial_t \varphi(t,x)\Bigl)\dx\dt}\\
&\leq \phi_1 \sn \sum_{K\in \TT}  \int_{t^n}^{t^{n+1}}\int_K
\Big| \partial_t \varphi(t,x_K)-\partial_t \varphi(t,x)\Big|\dx\dt
\end{split}\end{equation*}
 using that the  function  $\varphi$ is regular enough, we get
 \begin{equation}\label{ch4:C11h}
 \lim_{h \to 0}\abs{\mathfrak{C}_{1,1}^h-\mathfrak{C}_{1,1}^{h,*}}=0.
 \end{equation}
Similarly
\begin{equation*}
\begin{split}
\mathfrak{C}^h_{1,2}-\mathfrak{C}_{1,2}^{h,*}&=\sum_{K\in \TT}\int_K \phi_K s_0(x)(\varphi(0,x_K)-\varphi(0,x))\dx\\
&=\int_{\Omega}\phi_h s_0(x)(\varphi(0,x_K)-\varphi(0,x))\dx.
\end{split}
\end{equation*}
By the regularity of $\varphi$, there exists a positive constant $C$ such that\\
$\abs{\varphi(0,x_K)-\varphi(0,x)}\le C\,h$.
This implies
\begin{equation*}
\begin{split}
\abs{\mathfrak{C}_{1,2}^h-\mathfrak{C}_{1,2}^{h,*}}&\le C \, h \phi_1\sum_{K\in \TT}\int_K s_0(x)\dx.
\end{split}
\end{equation*}
Sending $h\to 0$ in the above inequality, we get
\begin{equation}\label{ch4:C12h}
\disp \lim_{h \to 0}\abs{\mathfrak{C}_{1,2}^h-\mathfrak{C}_{1,2}^{h,*}}=0.
\end{equation}

Combining (\ref{ch4:C11h}) with (\ref{ch4:C12h}), we obtain
\begin{equation}\label{ch4:conv:c1h}
 \lim_{h \to 0}\abs{\mathfrak{C}_{1}^{h}-\mathfrak{C}_{1}^{h,*}}=0,
 \end{equation}
but, $\mathfrak{C}_{1}^{h,*}$ can be written equivalently,
\begin{equation*}
\mathfrak{C}_1^{h,*} =\int_{Q_T} \phi_h s_h
\partial_t \varphi(t,x)\dx \dt-\int_{\Omega} \phi_h s^{0} \varphi(0,x)\dx.
\end{equation*}
Since   the bounded functions $ \phi_h $ and $ s_h$ converge almost everywhere respectively to $\phi$ and $ s$,
and as a consequence of Lebesgue dominated convergence theorem, we get
\begin{equation*}
\lim_{h \to 0}\mathfrak{C}_1^{h}=\lim_{h \to 0}\mathfrak{C}_1^{h,*} =\int_{Q_T} \phi s
\partial_t \varphi(t,x)\dx \dt-\int_{\Omega} \phi s^{0} \varphi(0,x)\dx.
\end{equation*}
Now, let us show that
\begin{equation}\label{ch4:conv:c2}
\lim_{h\to 0}\mathfrak{C}_2^h= \int_0^T\int_\O \nabla \beta(s)\cdot\nabla \varphi\;dxdt.
\end{equation}
Integrating by parts
\begin{align*}
\mathfrak{C}_2^h&=\frac 1 2\sn \Delta t \,
\sum_{K\in \TT} \sum_{L\in N(K)}
\abs{T_{K,L}}
l\,\frac{\beta(s^{n+1}_{L})-\beta(s^{n+1}_{K})}{d_{K,L}}\frac{\varphi(t^{n+1},x_L) - \varphi(t^{n+1},x_K)}{d_{K,L}}\\
& =
\frac 1 2\sn \Delta t \,
\sum_{K\in \TT} \sum_{L\in N(K)} \abs{T_{K,L}}\nabla_{K,L}\beta(s_h^{n+1})\cdot\eta_{K,l}\nabla\varphi(t^{n+1},x_{K,L})\cdot \eta_{K,L},
\end{align*}
where $x_{K,L}=\theta x_K+(1-\theta)x_L$, $0<\theta < 1$, is some point on the segment $]x_K,x_L[$. Recall that the value of $\nabla_{K,L}$ is directed by $\eta_{K,L}$, so
$$
\nabla_{K,L}\beta(s_h^{n+1})\cdot\eta_{K,l}\nabla\varphi(t^{n+1},x_{K,L})\cdot \eta_{K,L} = \nabla_{K,L}\beta(s_h^{n+1})\cdot \nabla\varphi(t^{n+1},x_{K,L}).
$$
since each term corresponding to the diamond $T_{K,L}$ appears twice,
$$
\mathfrak{C}_2^h = \int_0^T\int_\O \nabla_h \beta(s_h)\cdot(\nabla \varphi)_h\;dxdt,
$$
where
$$
(\nabla \varphi)_h|_{(t^n,t^{n+1}]\times T_{K,L}} :=  \nabla\varphi(t^{n+1},x_{K,L})
$$
Observe that from the continuity of $\varphi$ we get $(\nabla \varphi)_h\rightarrow \nabla \varphi $  in $L^\infty(Q_T)$. Hence the convergence (\ref{ch4:conv:c2})
is a consequence of (\ref{ch4:cvhpw}).\\
Now, we show the convergence of the flux,
\begin{equation}\label{ch4:conv:c3}
\lim_{h\to 0}\mathfrak{C}_3^h= -\int_0^T\int_\O M_2(s)\nabla p \cdot\nabla \varphi\;dxdt.
\end{equation}
Perform integration by parts (\ref{ch4:pppF}), thanks to the consistency of the fluxes, we obtain
$$ \mathfrak{C}_3^h=-\frac{1}{2}\sn \Delta t \sum_{K\in \TT}\sum_{L\in N(K)}
 G_2(s^{n+1}_K,s^{n+1}_{L};dp^{n+1}_{K,L})\Big( \varphi(t^{n+1},x_L)-\varphi(t^{n+1},x_K)\Big).$$
 For each couple of neighbours $K,L$ we denote $s^{n+1}_{K,L}$ the minimum of $s^{n+1}_K$ and $s^{n+1}_L$ and we introduce,

$$\mathfrak{C}_3^{h,*}=-\frac{1}{2}\sn \Delta t \sum_{K\in \TT}\sum_{L\in N(K)}
M_2(s^{n+1}_{K,L}) dp^{n+1}_{K,L}\Big( \varphi(t^{n+1},x_L)-\varphi(t^{n+1},x_K)\Big).$$
Define $\overline{s}_h$ and $\underline{s}_h$ by
$$\overline{s}_h|_{(t^n,t^{n+1}]\times T_{K,L}}:=\max\{s^{n+1}_K, s^{n+1}_L\},~~~~\underline{s}_h|_{(t^n,t^{n+1}]\times T_{K,L}}:=\min\{s^{n+1}_K, s^{n+1}_L\}.$$

Now, $\mathfrak{C}_3^{h,*}$ can be written under the following continues form,
$$\mathfrak{C}_3^{h,*}=-\int_0^T\int_\O M_2(\underline{s}_h)\nabla_h p_h \cdot(\nabla \varphi)_h\;dxdt.$$
By the monotonicity of $\beta$ and thanks to the estimate (\ref{ch4:est:grad-norm1}), we have
\begin{align*}
\int_0^T\int_\O |\beta(\overline{s}_h)-\beta(\underline{s}_h)|^2\,dxdt
&\leq \sn \Delta t \sum_{K\in \TT}\sum_{L\in N(K)}|T_{K,L}|\Big(\beta({s^{n+1}_L})-\beta({s^{n+1}_K})\Big)^2\\
&\leq Ch^2\sn \Delta t \sum_{K\in \TT}\sum_{L\in N(K)}\frac{|\sigma_{K,L}|}{d_{K,L}}|\beta({s^{n+1}_L})-\beta({s^{n+1}_K})|^2\\
&\leq Ch^2.
\end{align*}
Since $\beta^{-1}$ is continues, we deduce up to a subsequence,
\begin{equation}\label{ch4:conv:nbh}
|\overline{s}_h- \underline{s}_h|\rightarrow 0 \text{  a.e. on } Q_T.
\end{equation}
Moreover, we have; $\underline{s}_h\leq s_h \leq \overline{s}_h $ and $s_h\rightarrow s \text{ a.e. on } Q_T$. Consequently, and due to the continuity of the mobility function $M_2$ we have $M(\underline{s}_h)\rightarrow M(s) \text{ a.e. on } Q_T $  and in $L^p(Q_T)\text{ for } p<+\infty$. Using proposition \ref{ch4:propconvh} (\ref{ch4:cvhs}) and the strong convergence of $(\nabla \varphi)_h$ to $\nabla \varphi$, we obtain that
$$ \lim_{h\rightarrow 0} \mathfrak{C}_3^{h,*}=-\int_0^T\int_\O M_2(s)\nabla p \cdot\nabla \varphi\;dxdt.$$
It remains to show that
\begin{equation}\label{ch4:conv:c3*}
\lim_{h\rightarrow 0}| \mathfrak{C}_3^{h}-\mathfrak{C}_3^{h,*}|=0
\end{equation}
By the properties of the numerical flux function  (\ref{ch4:Hypfluxes}) we have
\begin{align*}
&|G_2(s^{n+1}_K,s^{n+1}_{L};dp^{n+1}_{K,L})- M_2(s^{n+1}_{K,L})dp^{n+1}_{K,L}|\\
=&|G_2(s^{n+1}_K,s^{n+1}_{L};dp^{n+1}_{K,L})- G_2(s^{n+1}_{K,L},s^{n+1}_{K,L};dp^{n+1}_{K,L})|\\
\leq &|dp^{n+1}_{K,L}|~~\omega(2|s^{n+1}_{L}-s^{n+1}_{K}|).
\end{align*}
Consequently,
$$ | \mathfrak{C}_3^{h}-\mathfrak{C}_3^{h,*}|\leq \int_0^T\int_\O \omega(2|s^{n+1}_{L}-s^{n+1}_{K}|)\nabla_hp_h\cdot(\nabla\varphi)_h\,dxdt $$
Applying the Cauchy-Schwarz inequality, and thanks to the uniform bound on $\nabla_hp_h$ and the convergence (\ref{ch4:conv:nbh}), we establish (\ref{ch4:conv:c3*}).
Now, we treat the convergence of the gravity term
\begin{equation}\label{ch4:conv:c4}
\lim_{h\to 0}\mathfrak{C}_4^h= -\int_0^T\int_\O \rho_2 M_2(s){\bf g} \cdot\nabla \varphi\;dxdt.
\end{equation}
Perform integration by parts (\ref{ch4:pppF}),
\begin{align*}
\mathfrak{C}_4^h&=\sn \Delta t \sum_{K\in \TT}\sum_{L\in N(K)}F^{n+1}_{2,K,L}\varphi(t^{n+1},x_K)\\
&=-\frac{1}{2}\sn \Delta t \sum_{K\in \TT}\sum_{L\in N(K)}F^{n+1}_{2,K,L}\Big(\varphi(t^{n+1},x_L) -\varphi(t^{n+1},x_K) \Big)
\end{align*}
We introduce,
\begin{multline*}
\mathfrak{C}_4^{h,*}=-\frac{1}{2}\sn \Delta t \sum_{K\in \TT}\sum_{L\in N(K)}M_2(s^{n+1}_{K,L})\int_{K/L}({\bf g}\cdot \eta_{K,L})\, d\gamma(x)\\
\Big(\varphi(t^{n+1},x_L) -\varphi(t^{n+1},x_K) \Big)
\end{multline*}
where $s^{n+1}_{K,L}:=\min \{s^{n+1}_{K},s^{n+1}_{L}\}$. We have
\begin{multline*}
\mathfrak{C}_4^{h}-\mathfrak{C}_4^{h,*}=\sn \Delta t \sum_{K\in \TT}\sum_{L\in N(K)}\Big( F^{n+1}_{2,K,L}- M_2(s^{n+1}_{K,L})\int_{K/L}({\bf g}\cdot \eta_{K,L})\, d\gamma(x)  \Big)\\
\Big(  \varphi(t^{n+1},x_L) -\varphi(t^{n+1},x_K) \Big)
\end{multline*}
relying on the assumption (H\ref{ch4:hyp:H3}), that  the mobility functions are Lipschitz, and  $\beta^{-1}$ is a Holder function, we deduce that
\begin{align*}
&\Big|( F^{n+1}_{2,K,L}- M_2(s^{n+1}_{K,L})\int_{K/L}({\bf g}\cdot \eta_{K,L})\, d\gamma(x) \Big|\\
&=\Big| \Big(M_2(s^{n+1}_{L})-M_2(s^{n+1}_{K,L})\Big)\int_{K/L}({\bf g}\cdot \eta_{K,L})^+\, d\gamma(x)\\
&-\Big(M_2(s^{n+1}_{K})-M_2(s^{n+1}_{K,L})\Big)\int_{K/L}({\bf g}\cdot \eta_{K,L})^-\, d\gamma(x)\Big| \\
&\leq C |{\bf g}| |\sigma_{K,L}| \Big| s^{n+1}_{L}-s^{n+1}_{K}\Big|\\
&\leq C |{\bf g}| |\sigma_{K,L}| \Big| \beta(s^{n+1}_{L})-\beta(s^{n+1}_{K})\Big|^\theta,
 \end{align*}
this yields to
\begin{align*}
|\mathfrak{C}_4^{h}-\mathfrak{C}_4^{h,*}|\leq C \sn \Delta t \sum_{K\in \TT}\sum_{L\in N(K)}|\sigma_{K,L}| d_{K,L}\big| \beta(s^{n+1}_{L})-\beta(s^{n+1}_{K})\big|^\theta
\big|\nabla_h \varphi_h  \big|.
\end{align*}
Using estimate \eqref{ch4:est:grad-norm1}, and Cauchy Schwarz inequality, then
$$
|\mathfrak{C}_4^{h}-\mathfrak{C}_4^{h,*}| \to 0 \text{ when }h\to 0.
$$
For the convergence of the source terms, $\mathfrak{C}_5^{h}+\mathfrak{C}_6^{h}$ can be written equivalently,
\begin{multline*}
\mathfrak{C}_5^{h}+\mathfrak{C}_6^{h}=\sn  \sum_{K\in \TT}\int_{t^n}^{t^{n+1}}\int_K  (s^{n+1}_K -1)f_P(t,x) \varphi(t^{n+1},x_K) \,dxdt\\
+\sn  \sum_{K\in \TT}\int_{t^n}^{t^{n+1}}\int_K f_I(t,x) \varphi(t^{n+1},x_K)   \,dxdt
\end{multline*}

Now, we introduce

\begin{multline*}
\mathfrak{C}_5^{h,*}+\mathfrak{C}_6^{h,*}=\sn  \sum_{K\in \TT}\int_{t^n}^{t^{n+1}}\int_K  (s^{n+1}_K -1)f_P(t,x) \varphi(t,x) \,dxdt\\
+\sn  \sum_{K\in \TT}\int_{t^n}^{t^{n+1}}\int_K f_I(t,x) \varphi(t,x) \,dxdt
\end{multline*}
 Due to regularity of $\varphi$ we obtain,
 $$
 \abs{\mathfrak{C}_5^{h}+\mathfrak{C}_6^{h}-\mathfrak{C}_5^{h,*}-\mathfrak{C}_6^{h,*}}\leq C (\Delta t +h)(\parallel f_P\parallel_{L^1(Q_T)}+\parallel f_I\parallel_{L^1(Q_T)})
 $$
 and this ensures,
 $$
 \lim_{h\to 0}\abs{\mathfrak{C}_5^{h}+\mathfrak{C}_6^{h}-\mathfrak{C}_5^{h,*}-\mathfrak{C}_6^{h,*}}=0.
 $$
We can write equivalently,
\begin{equation*}
\mathfrak{C}_5^{h,*}+\mathfrak{C}_6^{h,*}=\int_{Q_T} (s_h -1)f_P(t,x) \varphi(t,x) \,dxdt
+\int_{Q_T} f_I(t,x) \varphi(t,x) \,dxdt
\end{equation*}
Finally, by the convergence of the saturation function we get,
\begin{align*}
 \lim_{h\to 0}(\mathfrak{C}_5^{h}+\mathfrak{C}_6^{h})
 &=\lim_{h\to 0}(\mathfrak{C}_5^{h,*}+\mathfrak{C}_6^{h,*})\\
&=\int_{Q_T} (s -1)f_P(t,x) \varphi(t,x) \,dxdt
+\int_{Q_T} f_I(t,x) \varphi(t,x) \,dxdt
 \end{align*}

$\bullet$ Convergence of the discrete gas equation\\

Multiplying the discrete gas  equation  \eqref{ch4:prob-discr1} by $\Delta t\varphi(t^{n+1},x_K)$ for all
$K \in \TT$ and $n \in  \{0,\ldots,N\}$.
Summing the result over $K$ and $n$ yields
\begin{equation*}
\C_1^h+\C_2^h+\C_3^h+\C_4^h+\C_5^h=0,
\end{equation*}
where
\begin{equation*}\begin{split}
\C_1^h&=\sn \sum_{K\in \TT}\abs{K}\phi_K(\rho(p^{n+1}_K)s^{n+1}_K-\rho(p^{n}_K)s^{n}_K)\varphi(t^{n+1},x_K),\\
\C_2^h&=-\sn \Delta t \,
\sum_{K\in \TT} \sum_{L\in N(K)}
\frac{\abs{\sigma_{K,L}}}{d_{K,L}} \rho^{n+1}_{K,L}
(\beta(s^{n+1}_{L})-\beta(s^{n+1}_{K}))\varphi(t^{n+1},x_K),\\
\C_3^h&=\sn \Delta t \sum_{K\in \TT}\sum_{L\in N(K)}
 \rho^{n+1}_{K,L}G_1(s^{n+1}_K,s^{n+1}_{L};dp^{n+1}_{K,L})\varphi(t^{n+1},x_K),\\
\C_4^h&=\sn \Delta t \sum_{K\in \TT}\Big( \rho^2(p^{n+1}_K) M_1(s^{n+1}_K)\sum_{L \in N(K) }\int_{K/L}({\bf g}\cdot \eta_{K,L})^+\, d\gamma(x)\\
&- \rho^2(p^{n+1}_L) M_1(s^{n+1}_L)\sum_{L \in N(K) }\int_{K/L}({\bf g}\cdot \eta_{K,L})^-\, d\gamma(x)   \Big)\varphi(t^{n+1},x_K),\\
\C_5^h&=\sn \Delta t \sum_{K\in \TT}\abs{K}\rho(p^{n+1}_K)s^{n+1}_K f_{P,K}^{n+1}\varphi(t^{n+1},x_K).
\end{split}\end{equation*}

Performing integration by parts and keeping in mind
that $\varphi(T,x_K)=0$ for all $K \in \TT$, we obtain
\begin{equation*}
\begin{split}
\C_1^h=&-\sn \sum_{K\in \TT}\abs{K}\phi_K \rho(p^{n+1}_K)s^{n+1}_K(\varphi(t^{n+1},x_K)-\varphi(t^n,x_K))\\
&-\sum_{K\in \TT}\abs{K}\phi_K \rho(p^{0}_K)s^{0}_K\varphi(0,x_K)\\
=&-\sn \sum_{K\in \TT} \int_{t^n}^{t^{n+1}}\int_K \phi_K \rho(p^{n+1}_K)s^{n+1}_K \partial_t \varphi(t,x_K)\dx \dt\\
&-\sum_{K\in \TT}\int_K \phi_K \rho(p^{0}_K)s^{0}_K \varphi(0,x_K)\dx\\
=:&-\C^h_{1,1}-\C^h_{1,2}.
\end{split}
\end{equation*}
Let us also introduce
\begin{equation*}
\begin{split}
\C_1^{h,*} =&-\sn \sum_{K\in \TT} \int_{t^n}^{t^{n+1}}\int_K \phi_K \rho(p^{n+1}_K)s^{n+1}_K
\partial_t \varphi(t,x)\dx \dt\\
&-\sum_{K\in \TT}\int_K \phi_K \rho(p^{0}_K)s^{0}_K \varphi(0,x)\dx\\
=:&-\C_{1,1}^{h,*}-\C_{1,2}^{h,*}.
\end{split}
\end{equation*}
Then
\begin{equation*}
\begin{split}
\C^h_{1,2}-\C_{1,2}^{h,*}&=\sum_{K\in \TT}\int_K \phi_K \rho(p^{0}_K)s^{0}_K(\varphi(0,x_K)-\varphi(0,x))\dx.
\end{split}
\end{equation*}
By the regularity of $\varphi$, there exists a positive constant $C$ such that\\
$\abs{\varphi(0,x_K)-\varphi(0,x)}\le C\,h$.
This implies
\begin{equation*}
\begin{split}
\abs{\C_{1,2}^h-\C_{1,2}^{h,*}}&\le C \, h \abs{\Omega}.
\end{split}
\end{equation*}
Sending $h\to 0$ in the above inequality, we get
\begin{equation}\label{ch4:conv:c12h}
 \lim_{h \to 0}\abs{\C_{1,2}^{h}-\C_{1,2}^{h,*}}=0.
 \end{equation}

Now, due to the fact that the saturation function is bounded and the assumptions  (H\ref{ch4:hyp:H1}),(H\ref{ch4:hyp:H6})on the porosity and the density, we have
\begin{equation*}\begin{split}
\abs{\C_{1,1}^{h}-\C_{1,1}^{h,*}}
&= \abs{\sn \sum_{K\in \TT} \phi_K \rho(p^{n+1}_K)s^{n+1}_K \int_{t^n}^{t^{n+1}}\int_K
\Bigl( \partial \varphi(t,x_K)-\partial \varphi(t,x)\Bigl)\dx\dt}\\
&\leq \phi_1 \rho_M \sn \sum_{K\in \TT}  \int_{t^n}^{t^{n+1}}\int_K
\Big| \partial \varphi(t,x_K)-\partial \varphi(t,x)\Big|\dx\dt
\end{split}\end{equation*}
 using that the  function  $\varphi$ is regular enough, we get
 \begin{equation}\label{ch4:conv:c11h}
  \lim_{h \to 0}\abs{\C_{1,1}^h-\C_{1,1}^{h,*}}=0.
 \end{equation}
Combining (\ref{ch4:conv:c11h}) with (\ref{ch4:conv:c12h}), we obtain
\begin{equation}\label{ch4:conv:c1h}
 \lim_{h \to 0}\abs{\C_{1}^{h}-\C_{1}^{h,*}}=0
 \end{equation}
but, $\C_{1}^{h,*}$ can be written equivalently,
\begin{equation*}
\C_1^{h,*} =\int_{Q_T} \phi_h \rho(p_h)s_h
\partial_t \varphi(t,x)\dx \dt-\int_{\Omega} \phi_h \rho(p^{0}_h)s^{0}_h \varphi(0,x)\dx.
\end{equation*}
Since   $ \phi_h \rho(p_h)s_h$ and $\phi_h \rho(p^{0}_h)s^{0}_h$ converge almost everywhere respectively to $\phi \rho(p)s$ and $\phi \rho(p^{0})s^{0}$,
and as a consequence of Lebesgue dominated convergence theorem, we get
\begin{equation*}
\lim_{h \to 0}\C_1^{h}=\lim_{h \to 0}\C_1^{h,*} =\int_{Q_T} \phi \rho(p)s
\partial_t \varphi(t,x)\dx \dt-\int_{\Omega} \phi \rho(p^{0})s^{0} \varphi(0,x)\dx.
\end{equation*}
Now, let us show that
\begin{equation}
\lim_{h\to 0}\C_2^h= \int_0^T\int_\O \rho(p)\nabla \beta(s)\cdot\nabla \varphi\;dxdt.
\end{equation}
Integrating by parts
\begin{align*}
\C_2^h&=\frac 1 2\sn \Delta t \,
\sum_{K\in \TT} \sum_{L\in N(K)}
\abs{T_{K,L}} \rho^{n+1}_{K,L}
l\frac{\beta(s^{n+1}_{L})-\beta(s^{n+1}_{K})}{d_{K,L}}\frac{\varphi(t^{n+1},x_L) - \varphi(t^{n+1},x_K)}{d_{K,L}}\\
& =
\frac 1 2\sn \Delta t \,
\sum_{K\in \TT} \sum_{L\in N(K)} \abs{T_{K,L}}\rho^{n+1}_{K,L}\nabla_{K,L}\beta(s_h^{n+1})\cdot\eta_{K,l}\nabla\varphi(t^{n+1},x_{K,L})\cdot \eta_{K,L},
\end{align*}
where $x_{K,L}=\theta x_K+(1-\theta)x_L$, $0<\theta < 1$, is some point on the segment $]x_K,x_L[$. Recall that the value of $\nabla_{K,L}$ is directed by $\eta_{K,L}$, so
$$
\nabla_{K,L}\beta(s_h^{n+1})\cdot\eta_{K,l}\nabla\varphi(t^{n+1},x_{K,L})\cdot \eta_{K,L} = \nabla_{K,L}\beta(s_h^{n+1})\cdot \nabla\varphi(t^{n+1},x_{K,L}).
$$
since each term corresponding to the diamond $T_{K,L}$ appears twice,
\begin{equation}\label{ch4:c2h}
\C_2^h = \int_0^T\int_\O \rho(p_h)\nabla_h \beta(s_h)\cdot(\nabla \varphi)_h\;dxdt,
\end{equation}
where
$$
(\nabla \varphi)_h|_{(t^n,t^{n+1}]\times T_{K,L}} :=  \nabla\varphi(t^{n+1},x_{K,L})
$$
Define
\begin{equation}
D_2^h=-\sn \Delta t \,
\sum_{K\in \TT} \sum_{L\in N(K)}
\frac{\abs{\sigma_{K,L}}}{d_{K,L}} \overline{\rho}^{n+1}_{K,L}
(\beta(s^{n+1}_{L})-\beta(s^{n+1}_{K}))\varphi(t^n,x_K),
\end{equation}
where $\overline{\rho}^{n+1}_{K,L} = (\rho(P^{n+1}_{K}) + \rho(P^{n+1}_{L})/2$.
We have :
\begin{align*}
(\beta(s^{n+1}_{L})-\beta(s^{n+1}_{K})) \overline{\rho}^{n+1}_{K,L} &= (\beta(s^{n+1}_{L})\rho(P^{n+1}_{L})- \beta(s^{n+1}_{K})\rho(P^{n+1}_{K})) \\
&+ \beta(s^{n+1}_{L}) (\overline{\rho}^{n+1}_{K,L} - \rho(P^{n+1}_{L})) -
\beta(s^{n+1}_{K}) (\overline{\rho}^{n+1}_{K,L} - \rho(P^{n+1}_{K}))\\
&=(\beta(s^{n+1}_{L})\rho(P^{n+1}_{L})- \beta(s^{n+1}_{K})\rho(P^{n+1}_{K}))\\
&+(\beta(s^{n+1}_{L})+\beta(s^{n+1}_{K}) )( \rho(P^{n+1}_{K}) - \rho(P^{n+1}_{L}))/2.
\end{align*}
Then, $D_2^h$ can be rewritten
$$
D_2^h= D_3^h + D_4^h
$$
where
\begin{multline*}
D_3^h = \sn \Delta t \,
\sum_{\sigma_{K,L}\in E}\tau_{K,L}
 \Big(\beta(s^{n+1}_{L})\rho(p^{n+1}_{L})- \beta(s^{n+1}_{K})\rho(p^{n+1}_{K})\Big) \\
 \Big(\varphi(t^{n+1},x_K)  -\varphi(t^{n+1},x_L) \Big)
 \end{multline*}
 \begin{align*}
 D_4^h = \sn \Delta t \,
\sum_{\sigma_{K,L}\in E}\tau_{K,L}  \beta^{n+1}_{K,L} \Big( \rho(p^{n+1}_{K}) - \rho(p^{n+1}_{L})\Big) \Big(\varphi(t^{n+1},x_K)  -\varphi(t^{n+1},x_L) \Big)
\end{align*}

where $\beta^{n+1}_{K,L} = (\beta(s^{n+1}_{L})+\beta(s^{n+1}_{K}) )/2$,  recall that $\tau_{K,L} =\frac{\abs{\sigma_{K,L}}}{d_{K,L}}$. Follow the same lines as in \eqref{ch4:c2h},
\begin{align*}
&D_3^h =  \int_0^T\int_\O \nabla_h (\rho(s_h)\beta(s_h))\cdot(\nabla \varphi)_h\;dxdt,\\
& D_4^h = \int_0^T\int_\O \overline{\beta}(s_h)\nabla_h \rho(p_h)\cdot(\nabla \varphi)_h\;dxdt
\end{align*}
Using \eqref{ch4:est:grad-norm} and \eqref{ch4:est:grad-norm1}, we have
$$
\nabla_h (\rho(p_h)\beta(s_h)) \longrightarrow \nabla (\rho(s)\beta(s)) \text{ weakly } in L^2(Q_T),
$$
and using the fact that $(\nabla \varphi)_h$ converges strongly in $L^2(Q_T)$, we have
$$
D_3^h \longrightarrow   \int_0^T\int_\O \nabla (\rho(s)\beta(s))\cdot\nabla \varphi\;dxdt.
$$
In order to handle the convergence of $D_4^h$ we are going to show
\begin{equation}\label{ch4:betabar}
\overline{\beta}(s_h)\longrightarrow \beta(s) \text{ strongly in } in L^2(Q_T),
\end{equation}
and
\begin{equation} \label{ch4:gradrhoh}
\nabla_h \rho(p_h) \longrightarrow \nabla \rho^\star \text{ weakly in } in L^2(Q_T).
\end{equation}

The sequence $(\rho(p_h))_h$ is bounded then
\begin{equation}\label{ch4:rhostar}
\rho(p_h) \longrightarrow \rho^\star \text{ weakly in } L^2(Q_T).
\end{equation}
Using the fact that $\rho'(.)$ is bounded, we have
\begin{align*}
\|\nabla_h\rho(p_h)\|_{L^2(Q_T)}^2 &= \sn \Delta t \,\sum_{\sigma_{K,L}\in E}\tau_{K,L} | \rho(p^{n+1}_{L}))-\rho(p^{n+1}_{K}))|^2\\
& \le  \sn \Delta t \,\sum_{\sigma_{K,L}\in E}\tau_{K,L} |\rho'(p_{K,L}) (p^{n+1}_{L}-p^{n+1}_{K})|^2,
\end{align*}
and using the estimate \eqref{ch4:cvhs}, we deduce that $\nabla_h\rho(p_h)$ is bounded in $L^2(Q_T)$ and converges weakly to a function $\xi$ in $L^2(\Omega)$; and from \eqref{ch4:rhostar} we deduce that $\xi=\nabla \rho^\star$ weakly.

Recall that
$$
\beta(s_h)= \sn \sum_K \beta(s_K^{n+1})1_{K\times]t^n,t^{n+1}]} \longrightarrow \beta(s) \text{ strongly in } L^2(Q_T).
$$
Let us show for $\overline{\beta}(s_h)= \sn\sum_{\sigma_{K,L}\in E} \beta_{K,L}^{n+1}1_{T_{K,L}\times]t^n,t^{n+1}]}$,\\
 where $\beta_{K,L}^{n+1}=\frac{\beta(s_L^{n+1})+\beta(s_K^{n+1})}{2}$, that
$$
\overline{\beta}(s_h) -\beta(s_h) \longrightarrow 0 \text{ strongly in } L^2(Q_T)
$$
In fact,
\begin{align}
&\|\overline{\beta}(s_h) -\beta(s_h) \|^2_{L^2(Q_T)} \notag\\
&= \sn \Delta t \,\sum_{\sigma_{K,L}\in E}\big( | T_{K,L}\cap K| | \beta_{K,L}^{n+1}-\beta_{K}^{n+1}|^2 + | T_{K,L}\cap L| | \beta_{K,L}^{n+1}-\beta_{L}^{n+1}|^2\big) \notag\\
&\le  \frac 1 2 \sn \Delta t \,\sum_{\sigma_{K,L}\in E}  |T_{K,L}|(|\beta_{K}^{n+1}-\beta_{L}^{n+1}|^2),
\end{align}
and from estimate \eqref{ch4:est:grad-norm1}, there exists a positive constant $C$ such that
$$
|\overline{\beta}(s_h) -\beta(s_h) |^2_{L^2(Q_T)}  \le C h^2
$$
which establish the desired limit.
Then, the convergences \eqref{ch4:gradrhoh}, leads to
\begin{equation}\label{ch4:D_4}
\lim_{h\to 0}D_4^h = \int_0^T\int_\O \beta(s)\nabla \rho^\star\cdot\nabla \varphi\;dxdt.
\end{equation}
Finally, let us show $C_2^h-D_2^h\to 0$. We have
$$
C_2^h-D_2^h =\sn \Delta t \,
\sum_{\sigma_{K,L}} \tau_K,L(\tilde{\rho}^{n+1}_{K,L}
\big(\beta(s^{n+1}_{L})-\beta(s^{n+1}_{K})\big)\big(\varphi(t^{n+1},x_L) - \varphi(t^{n+1},x_K)\big),
$$
where $ \tilde{\rho}^{n+1}_{K,L} =  \rho^{n+1}_{K,L}- \overline{\rho}^{n+1}_{K,L}$.
This expression can be rewritten as
\begin{equation}\label{ch4:c2d2}
C_2^h-D_2^h = \int_{Q_T}\tilde{\rho}(p_h)\nabla_h\beta (s_h)\cdot (\nabla \phi)_h\;dxdt
\end{equation}
Let us show that $\tilde{\rho}(p_h)\to 0$ strongly in $L^2(Q_T)$.

We have
$$
\tilde{\rho}^{n+1}_{K,L} =  \rho^{n+1}_{K,L}- \overline{\rho}^{n+1}_{K,L} = \frac{1}{p_L^{n+1}-p_K^{n+1}}\int_{p_K^{n+1}}^{p_L^{n+1}} \rho(\psi)\,d\psi- \frac{\rho(p_K^{n+1})+\rho(p_L^{n+1})}{2},
$$
and from hypothesis  (H\ref{ch4:hyp:H6}), the function $\rho$ is monotone and uniformly Lipschitz, then there exists a positive such that
$$
\tilde{\rho}^{n+1}_{K,L} \le C |p_L^{n+1}-p_K^{n+1}|.
$$
So,
\begin{align}
\|\tilde{\rho}(p_h) \|^2_{L^2(Q_T)} &= \sn \Delta t \,\sum_{\sigma_{K,L}\in E}|T_{K,L}| | \tilde{\rho}^{n+1}_{K,L}|^2 \\
&\le  C\sn \Delta t \,\sum_{\sigma_{K,L}\in E}  |T_{K,L}|(|p_L^{n+1}-p_K^{n+1}|^2),
\end{align}
using \eqref{ch4:est:grad-norm}, we deduce
$$
\|\tilde{\rho}(p_h) \|^2_{L^2(Q_T)} \le C h^2,
$$
which goes to zero when $h$ goes to zero. This convergence combined with the weak convergence \eqref{ch4:cvhs} and the strong convergence of $(\nabla \phi)_h$  in $L^\infty(Q_T)$ shows that
$$
C_2^h-D_2^h \longrightarrow 0, \text{ when } h\to 0.
$$

\end{proof}

\begin{remark}\label{ch4:remarqueK}

In the case where ${\bf K}(x)$, the permeability tensor of the porous medium  at a point $x$, considered to be
$${\bf K}(x)=k(x)\mathcal{I}_d$$
 where $k$ is a scalar bounded  function of the space, $k(x)\ge k_0>0$ and $\mathcal{I}_d$  is the identity matrix. 
  The main part is the approximation of the dissipative terms (capillary terms) on each interface $\sigma_{K,L}$ as follows:

 Denote by,
 \begin{align*}
  k_{K}=\frac{1}{\abs{K}}\int_K k(x) \,dx.
  \end{align*}
  Now, we consider the following approximation
 \begin{align*}
 \int_{\sigma_{K,L}}k(x)\rho(p)\nabla \beta(s) \cdot \eta_{K,L}\, d\gamma &\approx
d^{*}_{K,L} \rho_{K,L} \frac{\abs{\sigma_{K,L}}}{d_{K,L}}(\beta(s_{L})-\beta(s_{K}))
\end{align*}

\begin{align*}
\int_{\sigma_{K,L}}k(x)\rho(p) M_1(s) \nabla p \cdot \eta_K \, d\gamma &\approx
d^{*}_{K,L} \rho_{K,L}\big(-M_1(s_L)(dp_{K,L})^{+} + M_1(s_K)(dp_{K,L})^-  \big)
\end{align*}

\begin{align*}
\int_{\sigma_{K,L}}k(x)\rho^2(p)M_1(s){\bf g}\cdot \eta_K\, d\gamma &\approx
d^{*}_{K,L} \Big(\rho^2(p^{n+1}_K) M_1(s^{n+1}_K){\bf g}_{K,L}
- \rho^2(p^{n+1}_L) M_1(s^{n+1}_L){\bf g}_{L,K}\Big)
\end{align*}
where,
 \begin{equation}\label{ch4:dharm}
 d^{*}_{K,L}=\frac{k_{K,L}\, k_{L,K}}{d(x_K,\sigma_{K,L})k_{K,L}+d(x_L,\sigma_{K,L})k_{L,K}} \,d(x_K,x_L),
 \end{equation}
 $$
dp_{K,L}= \frac{\abs{\sigma_{K,L}}}{d_{K,L}}\big( p_L-p_K   \big)=(dp_{K,L})^+ -  (dp_{K,L})^-,
$$
$${\bf g}_{K,L}:=\int_{K/L}({\bf g}\cdot \eta_{K,L})^+\, d\sigma=\int_{K/L}({\bf g}\cdot \eta_{L,K})^-\, d\gamma(x)
$$
and,
$$
\rho_{K,L}=
\begin{cases}\displaystyle
\frac{1}{p_L^{}-p_K^{}}\int_{p_K^{}}^{p_L^{}} \rho(\xi)\, d\xi & \text{ if } p_L^{}-p_K^{}\ne 0\\
\rho(p_K^{}) & \text{ otherwise }
\end{cases}
$$
In \eqref{ch4:dharm}, we take the harmonic average on the interfaces in order to ensure the conservation of numerical fluxes. 

\end{remark}



\nocite{*}

\end{document}